\documentclass[11pt,reqno]{amsart}

\usepackage{filecontents}
\usepackage{microtype}
\usepackage[english]{babel}
\usepackage[T1]{fontenc}
\usepackage[utf8]{inputenc}
\usepackage[T1]{fontenc}
\usepackage{lmodern}
\usepackage{latexsym}
\usepackage{verbatim}
\usepackage{hyperref}
\usepackage{amsthm}
\usepackage{amsmath}
\usepackage{amsfonts}
\usepackage[italian]{varioref}
\usepackage{imakeidx}
\usepackage{amssymb}
\usepackage{xcolor}
\usepackage{esint}
\usepackage{accents}
\usepackage{enumitem}
\usepackage{caption} 

\usepackage{soul}

\usepackage{tikz}
\usepackage{tikz,tikz-3dplot}
\usepackage{pgfplots}
\pgfplotsset{compat=1.16}
\tikzset{
    cross/.pic = {
    \draw[rotate = 45] (-#1,0) -- (#1,0);
    \draw[rotate = 45] (0,-#1) -- (0, #1);
    }
}
\usetikzlibrary{calc, angles}
\usetikzlibrary{arrows}
\usetikzlibrary{positioning}
\usetikzlibrary{intersections}
\usepackage{pgfplots}
\pgfplotsset{compat=newest}
\usepgfplotslibrary{fillbetween}
\usepackage{pgfplots}
\pgfplotsset{compat=newest}
\usepgfplotslibrary{fillbetween}
\usetikzlibrary{patterns}
\usepackage{mathtools}
\mathtoolsset{showonlyrefs}
\usepackage{amsmath}
\allowdisplaybreaks

	\definecolor{ao(english)}{rgb}{0.0, 0.5, 0.0}

\usepackage{lipsum}
\usepackage{curve2e}
\definecolor{gray}{gray}{0.4}

\usepackage{amsmath}
\usepackage{amsfonts}
\usepackage{amssymb}
\usepackage{graphicx}
\usepackage{hyperref} 
\usepackage{mathrsfs} 
\usepackage{relsize} 

\theoremstyle{plain}
\newtheorem{theorem}{Theorem}[section]

\newtheorem{definition}[theorem]{Definition}
\newtheorem{lemma}[theorem]{Lemma}
\newtheorem*{problem*}{Problem}

\newtheorem{proposition}[theorem]{Proposition}

\newtheorem*{theorem*}{Theorem}

\theoremstyle{definition}

\newtheorem{remark}[theorem]{Remark}

\def\Item$#1${\item $\displaystyle#1$
   \hfill\refstepcounter{equation}(\theequation)}
   
\theoremstyle{plain} 

\theoremstyle{definition}

\theoremstyle{remark} 
\numberwithin{equation}{section}
\usepackage[makeroom]{cancel}

\usepackage{bm} 
\usepackage{graphicx}
\makeindex
\usepackage{mathtools}
\usepackage{slashed}
\date{}

\usepackage{tikz,tikz-3dplot}
\tdplotsetmaincoords{80}{45}
\tdplotsetrotatedcoords{-90}{180}{-90}
\usetikzlibrary{arrows, automata, intersections}

\newcommand{\rr}{\mathbb{R}}
\newcommand{\hh}{\mathbb{H}}

\newcommand{\sss}{\mathbb{S}}
\newcommand{\e}{\epsilon}

\newcommand{\inn}{\textnormal{in}\ }
\newcommand{\onn}{\textnormal{on}\ }
\newcommand{\andd}{\textnormal{and}}

\newcommand{\dd}{\textnormal{d}}

\newcommand{\norm}[2]{\left|\left| #1 \right| \right|_{#2}}

\newcommand{\scal}[1]{ \left( #1 \right)_{\mathcal{L}^2} }

\newcommand{\pder}[2]{\frac{\partial #1}{\partial #2}}

\counterwithout{equation}{subsubsection}
\numberwithin{equation}{section}
\tikzstyle{mybox} = [draw=black, very thick, rectangle, rounded corners, inner ysep=5pt, inner xsep=5pt]

\title[Serrin domains of $\sss^3$]{Non-isoparametric Serrin domains of $\sss^3$ with connected toric boundary}
\author{Andrea Bisterzo and Shigeru Sakaguchi}

\address{Andrea Bisterzo
  \newline \indent Centro di Ricerca Matematica Ennio De Giorgi
  \newline \indent
Scuola Normale Superiore
\newline\indent Piazza dei Cavalieri 3, 56126, Pisa, Italy.}
\email{andrea.bisterzo@sns.it}

\address{Shigeru Sakaguchi
  \newline \indent Admissions Center
  \newline \indent
Tohoku University
\newline\indent Sendai, 980-8576, Japan.}
\email{sigersak@tohoku.ac.jp}

\begin{document}


\begin{abstract}
We investigate the overdetermined  torsion problem
\begin{align}
\begin{cases}
-\Delta u = 1 & \inn \Omega\\ u=0 & \onn \partial \Omega\\ \pder{u}{\nu}=\textnormal{const.} & \onn \partial \Omega,
\end{cases}
\end{align}
where $\Omega$ is a smooth Riemannian domain. Domains admitting a solution to this problem are called \textit{Serrin domains}, after the celebrated work of Serrin \cite{Se71}, where is proved that in $\rr^n$ such domains are geodesic balls.
In the present paper we establish the existence of two distinct types of Serrin domains of $\sss^3$, respectively of small and large volume, 
each of whose boundary is connected and is neither isometric to a geodesic sphere nor to a Clifford torus. These domains arise as nontrivial perturbations of some classical symmetric solutions to the same problem. 
Our approach relies on an implicit construction based on the Crandall-Rabinowitz bifurcation theorem, which allows us to detect branches of non-radial solutions bifurcating from a family of radial ones. 
The resulting examples highlight new geometric configurations of the torsion problem in the three-dimensional sphere, providing another proof of the fact that the rigidity of Serrin-type results can fail in the presence of curvature.
\end{abstract}

\maketitle
\tableofcontents


\section{Introduction}
Let $(M,g)$ be a smooth Riemannian manifold and $\Omega\subset M$ a bounded, smooth domain with outer unit normal field $\nu$. Denote by $\Delta$ the negative, in the spectral sense, Laplace-Beltrami operator of $M$ (i.e., in $\rr$ we have $\Delta=\pder{^2}{x^2}$). Assuming the existence of a solution to the following boundary value problem
\begin{align}\label{Eq:Torsion}
\begin{cases}
-\Delta u = 1 & \inn \Omega\\ u=0 & \onn \partial \Omega\\ \pder{u}{\nu}= \textnormal{const.} & \onn \partial \Omega,
\end{cases}
\end{align}
in the seminal paper \cite{Se71}, J. Serrin proved that if $M=\rr^n$ then the domain $\Omega$ is forced to be a geodesic ball, hence providing one of the most fascinating and celebrated rigidity results in geometric analysis.
\medskip

\begin{definition}
A domain $\Omega \subset M$ supporting a solution to \eqref{Eq:Torsion} is said to be a \textnormal{Serrin domain} (or \textnormal{harmonic domain}). 
\end{definition}

\begin{remark}In \cite{FM15} the authors proved that (possibly small) Serrin domains exist in any compact Riemannian manifold: they are constructed by perturbing sufficiently small geodesic balls.
\end{remark}
\medskip

Serrin's result, proved originally by using the Alexandrov's moving planes method (for an alternative proof see also the well known work by H.F. Weinberger \cite{We71}), has been widely studied in subsequent years. 
Remarkable extensions to other ambient manifolds have been obtained by R. Molzon in \cite{Molzon91} and by S. Kumaresan and J. Prajapat in \cite{KP98}. 
Here the authors, with the same technique, provided a rigidity results for Serrin domains contained in the hyperbolic space $\hh^n$ or in the hemisphere $\sss^n_+$. Other significant extension (and generalizations) to different ambient manifolds and/or to other differential equations have been obtained. A comprehensive account is beyond the scope of this paper: we only refer the reader to \cite{ABBM25, ABM24, AFS25, BNST08, CV19, FM15, FR22, FRS24, Ro18} and the references cited therein.

Among the above mentioned works, the rigidity theorems in $\sss^n_+$ due to \cite{Molzon91, KP98} opened several questions. 
The fact that the result obtained only holds for domains contained in a half-sphere seemed purely technical at the beginning, giving the perception that it could only be a pathology of the method the authors used. This opened the path to new and extremely deep problems. A first question is
\medskip
\begin{center}
\begin{enumerate}
\item[1.] \textit{Is any Serrin domain of $\sss^n$ a geodesic ball?}
\end{enumerate}
\end{center}
\medskip
The answer is clearly false. Indeed, one can easily construct infinitely many Serrin domains with non-connected boundary just considering any symmetric tubular neighbourhood of the equator $\sss^{n-1}\hookrightarrow \sss^n$. We stress that, in this case, the boundary components of the domains are geodesic spheres. This naturally led to the following question
\medskip
\begin{center}
\begin{enumerate}
\item[2.] \textit{Is any Serrin domain of $\sss^n$ bounded by geodesic spheres?}
\end{enumerate}
\end{center}
\medskip
If one does not impose constrictions on the topology of the boundary of the domain, then the answer to this second question is, again, trivial: no. Indeed, in the particular case $n=3$, any symmetric tubular neighbourhood of the Clifford torus $\sss^1\left( \frac{1}{\sqrt{2}}\right)\times \sss^1\left( \frac{1}{\sqrt{2}}\right) \hookrightarrow \sss^3$ is a Serrin domain. For higher dimensions, similar examples can be constructed using Cartan-Munzner polynomials (see \cite{BP23, Sh00}). As a consequence, one can justifiably ask
\medskip
\begin{center}
\begin{enumerate}
\item[3.] \textit{Is any Serrin domain of $\sss^n$ with simply connected boundary components bounded by geodesic spheres?}
\end{enumerate}
\end{center}
\medskip
A (negative) answer to this second question has been provided by M. M. Fall, I. A. Minlend and T. Weth in the celebrated paper \cite{FMW}. Here the authors constructed a non-symmetric (i.e. not bounded by geodesic spheres) tubular neighbourhood of the equator supporting a solution to \eqref{Eq:Torsion}. The proof is very technical and relies on a suitable version of the Crandall-Rabinowitz bifurcation theorem, \cite{CR71}. The paper \cite{FMW} represents a milestone in the understanding of the validity of Serrin-like theorems in more general manifolds. For similar results, we also refer to \cite{DVEPS19, DVEPS23}. Starting from this, we reached the following problem
\medskip
\begin{center}
\begin{enumerate}
\item[4.] \textit{Is any Serrin domain of $\sss^n$ with connected boundary a geodesic ball?}
\end{enumerate}
\end{center}
\medskip
Again, the answer is negative, at least in $\sss^3$. Indeed, let $\mathcal{F}=\{\Sigma_t\}_{t\in \left[0,\frac{\pi}{2}\right]}$ be a singular Riemannian foliation of $\sss^3$ of Clifford tori
\begin{align}
\Sigma_t=\sss^1\left(\sin(t)\right)\times \sss^1\left(\cos(t)\right)\hookrightarrow \sss^3.
\end{align}
Fixed any $T\in \left(0,\frac{\pi}{4}\right)$, the domain $\Omega_T= \sqcup_{0\leq t< T} \Sigma_t$, that is, the geodesic ball of radius $T$ around $\Sigma_0$, is an isoparametric domain (i.e. foliated by parallel CMC surfaces) with homogeneous leaves. By a stability argument (see \cite[Theorem 5.1]{BP23} or \cite{Bi24}) one can prove that any solution to the torsion problem
\begin{align}
\begin{cases}
-\Delta u=1 & \inn \Omega_T\\
u=0 & \onn \partial \Omega_T
\end{cases}
\end{align}
is, in fact, radial, i.e. is constant on each $\Sigma_t$. This, in turn, implies that $\pder{u}{\nu}$ is constant on $\partial \Omega_T$, showing that $\Omega_T$ is a Serrin domain with connected boundary and not isometric to a geodesic ball. We stress that $\Omega_T$ is not contained in any hemisphere, even for small $T$. Hence, the previous claim is not in contradiction with \cite{Molzon91, KP98}. The existence of these other kinds of Serrin domains carried us to ask ourselves
\medskip
\begin{center}
\begin{enumerate}
\item[4'.] \textit{Is any Serrin domain of $\sss^3$ with connected boundary either a geodesic ball or an isoparametric domain with toric boundary?}
\end{enumerate}
\end{center}
\medskip
The aim of the present paper is to answer, again in negative, to the above point. Our main theorem is the following:


\begin{theorem}\label{Thm:Main}
There exist two families of Serrin domains, each obtained by perturbing a certain $\Omega_T$, whose volumes are respectively close to $0$ and to $|\sss^3|$ and that satisfy the following properties
\begin{enumerate}
\item their boundaries are connected;
\item their boundaries are neither isometric to a geodesic sphere nor to a Clifford torus.
\end{enumerate}
\end{theorem}

\begin{remark}
What seen so far led to another deep question:
\begin{center}
\begin{enumerate}
\item[5.] \textit{Is any Serrin domain of $\sss^n$ with connected and simply connected boundary a geodesic ball?}
\end{enumerate}
\end{center}
A possible answer is suggested by the very interesting work \cite{RSW}. Here D. Ruiz, P. Sicbaldi and J. Wu constructed some domains with connected and simply connected boundary that are non-isometric to geodesic balls and that support a positive solution to the homogeneous Dirichlet problem associated to the equation $-\Delta u =\frac{u^p-u}{\e}$, for certain positive parameters $p$ and $\e$, and having constant normal derivative on $\partial \Omega$. The result is obtained by combining, in a very delicate way, a stereographic projection with a suitable bifurcation theorem (in this case the one by Krasnoselskii). Even if the differential equation considered is different from that of \eqref{Eq:Torsion}, this result seems to suggest that the answer to the question 5 might be \textit{No}.
\end{remark}


\subsection{Main tool and structure of the paper}
The strategy for proving Theorem \ref{Thm:Main} is similar to the one adopted in \cite{FMW} and is based on the following version of the Crandall-Rabinowitz bifurcation theorem, stated in this form in \cite[Theorem 6.1]{FMW} and equivalent to the original one \cite[Theorem 1.7]{CR71}.

\begin{theorem}[Crandall-Rabinowitz]\label{Thm:CRBifurcation}
Let $(A,\norm{\cdot}{A})$ and $(B,\norm{\cdot}{B})$ be two Banach spaces, $\mathcal{O}\subset \rr \times A$ an open set and $I\subset \rr$ an open interval so that $I\times \{0\}\subset \mathcal{O}$. Denote by $(\lambda,\phi)$ the elements of $ \rr \times A$. Consider $G\in C^2(\mathcal{O};B)$ and suppose there exist $\lambda_*\in I$ and $a_*\in A\setminus\{0\}$ so that
\begin{enumerate}
\item $G(\lambda,0)=0$ for all $\lambda \in I$;
\item $\textnormal{ker}(D_\phi G(\lambda_*,0))=\textnormal{Span}_\rr\{a_*\}$;
\item $\textnormal{codim}(\textnormal{Im}(D_\phi G(\lambda_*,0)))=1$;
\item $D_\lambda D_\phi G(\lambda_*,0)(a_*) \notin \textnormal{Im}(D_\phi G(\lambda_*,0))$.
\end{enumerate}
Then, for any complement $\mathcal{C}$ of the space $\textnormal{Span}_\rr\{a_*\}$ in $A$ there exists a continuous curve
\begin{align}
(-\e,\e) & \to \left(\rr \times \mathcal{C}\right)\cap \mathcal{O}\\
s & \mapsto (\lambda(s),w(s))
\end{align}
such that
\begin{enumerate}
\item $\lambda(0)=\lambda_*$ and $w(0)=0$;
\item $s(a_*+w(s)) \in \mathcal{O}$ and $\lambda(s) \in I$;
\item $G\Big(\lambda(s), s(a_*+w(s))\Big)=0$;
\item there exists a neighbourhood of $(\lambda_*,0)$ in $\rr \times A$ where every solution to $G(\lambda,\phi)=0$ is either in $\left\{\Big(\lambda(s), s(a_*+w(s))\Big)\ :\ s\in (-\e,\e)\right\}$ or in  $\rr\times \{0\}$.
\end{enumerate}
\end{theorem}

The paper is divided as follows. Section \ref{Sec:NotationPreliminaries} is devoted to define the notation we adopt in the whole paper and to introduce the functional operator $H$ and its linearization $\mathbb{L}$, needed for the bifurcation argument we are going to present. In Section \ref{Sec:NotationPreliminaries} we also study some spectral properties of the linearization $\mathbb{L}$ of $H$, exhibiting its spectrum in some specific cases. We stress that Section \ref{Sec:NotationPreliminaries} is strongly inspired by \cite{FMW}. Section \ref{Sec:EigenvaluesProperties} presents mainly the most delicate part, which will be crucial to the construction of the non-isoparametric Serrin domains whose volumes are close to $0$: the study of the spectrum of $\mathbb{L}$ under specific perturbations of the boundary of $\Omega_T$. To conclude, in Section \ref{Sec:Bifurcation} we introduce the functional setting we need in order to be able to apply the bifurcation argument. In Appendix \ref{Sec:AppA} we include some technical tools, originally proved in \cite{FMW} in a different setting and whose proofs are a straightforward adaptation of the original ones to this new setting. Appendix \ref{Sec:AppB} concerns the crucial part to construct the non-isoparametric Serrin domains whose volumes are close to $|\sss^3|$.



\section{Notation and preliminaries}\label{Sec:NotationPreliminaries}

The present section is aimed to lay the foundation to the geometric and functional settings we are going to deal with.

\subsection{Setting}
We start by considering the following parametrization of $\sss^3$
\begin{align}\label{Count_Param_1_KS_OLD}
\Upsilon: \left[0, \frac{\pi}{2}\right]\times \sss^1_\eta \times \sss^1_\xi & \to \sss^3\\
(\theta, \eta, \xi) &\mapsto \left(\begin{array}{c} \sin(\theta) \cos(\eta)\\ \sin(\theta) \sin(\eta)\\ \cos(\theta) \cos(\xi)\\ \cos(\theta) \sin(\xi) \end{array} \right),
\end{align}
looking at $\sss^1$ as $(-\pi,\pi]/_\sim$, where $\sim$ identifies $-\pi$ and $\pi$. We stress that the above parametrization provides an isoparametric foliation of $\sss^3$ by Clifford tori $\Sigma_t:=\Upsilon(t,\sss^1,\sss^1)$ that are parallel to each other.

Fix a smooth function $\phi:\sss^1\times \sss^1 \to \rr$ so that $0< \phi < \frac{\pi}{2}$ and define
\begin{align}
\Omega_\phi:=\left\{(\theta,\eta,\xi)\in \rr \times \sss^1_\eta \times \sss^1_\xi \ :\ \theta \in \Big[0, \phi(\eta,\xi) \Big)\right\}\big/ \sim,
\end{align}
where $\sim$ identifies $\{0\}\times \sss^1_\eta$ with a single point $o$, and hence $\{0\}\times \sss^1_\eta \times \sss^1_\xi$ is identified with a one-dimensional manifold $\Sigma_0=\{o\}\times \sss^1_\xi$. The domain $\Omega_\phi$ is naturally endowed with the pull-back metric $g=\Upsilon^*g^{\sss^3}$
\begin{align}
g_{(\theta, \eta, \xi)}=\textnormal{d}\theta^2+\sin^2(\theta)\textnormal{d}\eta^2+ \cos^2(\theta) \textnormal{d}\xi^2,
\end{align}
smoothly extended at $\Sigma_0$ (see \cite[Section 1.4.1, Proposition 1]{Pe}). The associated Laplace-Beltrami operator at a fixed point $(\theta, \eta, \xi)$ writes as
\begin{align}
\Delta^g_{(\theta,\eta,\xi)}&=\frac{\partial^2}{\partial \theta^2} + \left[\frac{\cos(\theta)}{\sin(\theta)}-\frac{\sin(\theta)}{\cos(\theta)} \right] \frac{\partial}{\partial \theta} + \frac{1}{\sin^2(\theta)} \frac{\partial^2}{\partial \eta^2}+\frac{1}{\cos^2(\theta)} \frac{\partial^2}{\partial \xi^2}.
\end{align}
In the particular case $\phi \equiv \lambda \in \left( 0,\frac{\pi}{2}\right)$ is constant, we get
\begin{align}
\Omega_{\lambda}= \left(\left[0, \lambda\right)\times \sss^1_\eta \times \sss^1_\xi\right)\big/ \sim
\end{align}
that through the map $\Upsilon$ (observe that $\Upsilon$ is well defined on the quotient space $\Omega_\lambda$) provides the ball of radius $\lambda$ around $\sss^1\hookrightarrow \sss^3$
\begin{align}
\Upsilon(\Omega_\lambda)=B_\lambda \Big(\sss^1 \Big).
\end{align}
Next steps consist in moving any information concerning the fixed function $\phi$ from the domain $\Omega_\phi$ into another Riemannian metric defined on a domain that does not depend on $\phi$. To this aim, we consider the map
\begin{align}\label{Eq_ParamPsi_KS_OLD}
\Psi^\phi:\Omega:=\left([0,1)\times \sss^1_\eta\times \sss^1_\xi\right)\Big/\sim & \to \Omega_\phi\\
(t,\eta,\xi)&\mapsto (t\phi(\eta,\xi),\eta,\xi),
\end{align}
where $\sim$ still denotes the identification of $\{0\}\times \sss^1_\eta$ with a point $o$. The map $\Psi^\phi$ defines a diffeomorphism between $\Omega_\phi$ and $\Omega$. In the case $\phi\equiv \lambda$ is constant, the pull-back metric $g^\lambda=(\Psi^\lambda)^* g$ is given by
\begin{align}
g^\lambda:=(\Psi^\lambda)^* g = \lambda^2 \dd t^2 + \sin^2(t\lambda) \dd \eta^2 + \cos^2(t\lambda) \dd \xi^2
\end{align}
and the associated Laplace-Beltrami operator is
\begin{align}
\Delta^{g^\lambda}&= \lambda^{-2} \pder{^2}{t^2} + \lambda^{-1} \left[\frac{\cos(t\lambda)}{\sin(t\lambda)}-\frac{\sin(t\lambda)}{\cos(t\lambda)} \right]\pder{}{t}+\frac{1}{\sin^2(t\lambda)} \pder{^2}{\eta^2}+\frac{1}{\cos^2(t\lambda)} \pder{^2}{\xi^2}.
\end{align}

Observe that the boundary of $\Omega$ has one connected component, which can be seen as the level set of the function
\begin{align}
f:\Omega\to \rr\\
(t,\eta,\xi)&\mapsto t.
\end{align}
Hence, the outward pointing normal field to $\partial \Omega$ can be defined as
\begin{align}
\nu_\phi(1,\eta,\xi) &:= \frac{\nabla^{g^\phi} f(1,\eta,\xi)}{||\nabla^{g^\phi} f(1,\eta,\xi)||_{g^\phi}}.
\end{align}
If $\phi\equiv \lambda$ is constant, it reduces to
\begin{align}
\nu_\lambda (1,\eta,\xi)&=\frac{1}{\lambda} \frac{\partial}{\partial t}\Big|_{(1,\eta,\xi)}.
\end{align}
By construction, in the parametrization $(\Omega_\phi, g)$ the outward pointing vector field $\mu_\phi$ normal to $\partial \Omega_\phi$ is given by
\begin{align}
\mu_\phi=\textnormal{d}\Psi^\phi [\nu_\phi]
\end{align}
that reduces to
\begin{align}
\mu_\lambda(\lambda,\eta,\xi)=\frac{\partial}{\partial \theta}\Big|_{(\lambda,\eta,\xi)} \quad \onn \partial \Omega_\lambda
\end{align}
if $\phi\equiv \lambda\in \left( 0,\frac{\pi}{2}\right)$ is constant.


\subsection{Operator $H$}

Fix a suitable $\phi$ as above and consider the solution $v$ to the torsion problem
\begin{align}\label{Eq:vTorsion}
\begin{cases}
-\Delta^g v=1 & \inn \Omega_\phi \\ v=0 & \onn \partial \Omega_\phi.
\end{cases}
\end{align}
If $\phi\equiv \lambda$ is constant, since $v$ is stable and $\Omega_\lambda$ is an isoparametric domain, it follows that the function $v$ is radial, in the sense that it only depends on the variable $\theta$ (see \cite{BP23, Bi24}). In particular, $v$ is explicit
\begin{align}
v(\theta, \eta, \xi)&=v(\theta)\\
&=\widetilde{v}(\theta)-\widetilde{v}(\lambda)\\
&=\frac{1}{2} \ln\left(\frac{\cos(\theta)}{\cos(\lambda)} \right),
\end{align}
where $\widetilde{v}(\theta)=\frac{1}{2} \ln\left(\cos(\theta)\right)$ satisfies
\begin{align}
-\Delta^g \widetilde{v}=1 \quad \inn \left(0,\frac{\pi}{2} \right)\times \sss^1_\eta \times \sss^1_\xi.
\end{align}
Moreover
\begin{align}
\pder{v}{\mu_\lambda}=\pder{v}{\theta}=-\frac{1}{2} \tan(\lambda) \quad \onn \partial \Omega_\lambda,
\end{align}
implying that $\Omega_\lambda$ is a Serrin domain.

Now we introduce an operator that will play a crucial role in our bifurcation argument. To this aim, consider the torsion problem in $\Omega$
\begin{align}\label{Eq:ulTorsion}
\begin{cases}
-\Delta^{g^\phi} u_\phi=1 & \inn \Omega\\
u_\phi=0 & \onn \partial \Omega.
\end{cases}
\end{align}
Since \eqref{Eq:vTorsion} is invariant with respect to isometries, it follows that $u_\phi$ is the pull-back of $v$ through $\Psi^\phi$: $u_\phi=(\Psi^\phi)^*v$. Hence, if $\phi\equiv \lambda$ is constant, then
\begin{align}
\pder{u_\lambda}{\nu_\lambda}=\textnormal{const.} \quad \onn \partial \Omega,
\end{align}
where the constant can be explicitly computed and turns out to be $-\frac{|\Omega|}{|\partial \Omega|}$. 

We define the operator $H:\mathcal{U}\subset C^{2,\alpha}(\sss^1\times \sss^1)\to C^{1,\alpha}(\sss^1\times \sss^1)$ as
\begin{align*}
H:\phi \mapsto \frac{\partial u_\phi}{\partial \nu_\phi}(1, \cdot, \cdot),
\end{align*}
where $\mathcal{U}\subseteq C^{2,\alpha}(\sss^1\times \sss^1)$ is a subspace to be fixed and $0 < \alpha < 1$.

Then the uniqueness of the solution $u_\phi$ of problem \eqref{Eq:ulTorsion} gives


\begin{lemma}\label{Lem_evenness}
The following properties hold:
\begin{enumerate}
\item If $\phi=\phi(\eta,\xi)$ is even in $\xi \in \sss^1 (= (-\pi,\pi]/_\sim)$, then both $u_\phi$ and $H(\phi)$ are as well;
\item If $\phi$ depends only on $\xi \in \sss^1$, then both  $u_\phi$ and $H(\phi)$ do;
\item The same propositions as (1) and (2) hold true if $\xi \in \sss^1$ is replaced with
$\eta \in \sss^1$.
\end{enumerate}
\end{lemma}

If $\phi\equiv \lambda$ is constant, then $H$ is more specific as follows.


\begin{lemma}\label{Lem_RadialSolution_KS_OLD}
If $\phi\equiv \lambda$ is constant, then
\begin{align*}
u_\lambda(t,\eta,\xi)=(\Psi^{\lambda})^* v(\theta)= v(\theta)=v(t\lambda)
\end{align*}
and
\begin{align}
H(\lambda)(\eta,\xi)=\pder{v}{\theta}(\lambda)=-\frac{1}{2} \tan(\lambda).
\end{align}
\end{lemma}
\begin{proof}
The first claim has already been observed. For the second claim we notice that
\begin{align}
H (\lambda)(\eta,\xi)=&g^\lambda\left(\nabla^{g^\lambda} u_\lambda, \nu_\lambda \right)_{(1,\eta,\xi)}\\
=&g^\lambda\left(\lambda^{-2} \frac{\partial}{\partial t} \left(v\left(t\lambda\right) \right)\frac{\partial}{\partial t}, \lambda^{-1} \frac{\partial}{\partial t}\right)_{(1,\eta,\xi)}\\
=& \pder{v}{\theta} (\lambda) g^\lambda \left(\lambda^{-1} \frac{\partial}{\partial t},\lambda^{-1} \frac{\partial}{\partial t} \right)_{(1,\eta,\xi)}\\
=& \pder{v}{\theta} (\lambda).
\end{align}
\end{proof}


\subsection{Linearization and spectrum}

The last step of this section is the computation of the linearization of the operator $H$. To this aim, fix $w\in C^{2,\alpha}(\sss^1\times \sss^1)$ and consider the $g^\lambda$-harmonic extension $\varphi^\lambda$ of $w$ in $\Omega$, i.e. the solution to
\begin{align}\label{Eq_HarmonicExtension_KS_OLD}
\begin{cases}
-\Delta^{g^\lambda} \varphi^\lambda =0 & \inn \Omega\\
\varphi^\lambda(1,\eta,\xi)=w(\eta,\xi) & \onn \partial \Omega.
\end{cases}
\end{align}
In the particular case in which $w(\eta,\xi)=A(\eta)B(\xi)$ is the product of two eigenfunctions $A$ and $B$ of $-\Delta^{\sss^1}$, the solution to the above problem can be expressed as
\begin{align}
\varphi^\lambda (t,\eta,\xi)=l^{\epsilon,\delta,\lambda}(t) w(\eta,\xi)
\end{align}
where, denoting by $\epsilon^2$ and $\delta^2$ the eigenvalues of $A$ and $B$ respectively, the function $l^{\epsilon,\delta,\lambda}$ satisfies
\begin{align}\label{Eq:l_1}
\begin{cases}
(l^{\epsilon,\delta,\lambda})''(t) + \lambda \left[\frac{\cos(t\lambda)}{\sin(t\lambda)}-\frac{\sin(t\lambda)}{\cos(t\lambda)}\right] (l^{\epsilon,\delta,\lambda})'(t) \\
\quad \quad \quad \quad - \lambda^{2}\left[\frac{\epsilon^2}{\sin^2(\lambda t)} + \frac{\delta^2}{\cos^2(\lambda t)}\right] l^{\epsilon,\delta,\lambda} (t)=0 & \inn (0,1) \\
l^{\epsilon,\delta,\lambda} (1)=1.
\end{cases}
\end{align}
The existence of such solutions follows from standard ODEs theory.

Next proposition is one of the crucial points. We postpone the proof to Appendix \ref{Sec:AppA}, since it is a slight modification of \cite[Proposition 2.4]{FMW}.


\begin{proposition}\label{Prop_Linearization_KS_OLD}
For any $\lambda\in \left(0, \frac{\pi}{2}\right)$ the linear operator
\begin{align}
\mathbb{L}_\lambda:=D_\phi \Big|_{\phi\equiv \lambda} H :C^{2,\alpha}(\sss^1\times \sss^1)\to C^{1,\alpha}(\sss^1\times \sss^1)
\end{align}
is given by
\begin{align}\label{Eq_Prop_Linearization_KS_OLD}
\mathbb{L}_\lambda[w](\eta,\xi)&=-\pder{\widetilde{v}}{\theta}(\lambda)\frac{1}{\lambda} \pder{\varphi^\lambda}{t}(1,\eta,\xi) +\pder{^2\widetilde{v}}{\theta^2}(\lambda) w(\eta, \xi)\\
&=\frac{\tan(\lambda)}{2\lambda} \pder{\varphi^\lambda}{t}(1,\eta,\xi) - \frac{1}{2} \frac{1}{\cos^2(\lambda)} w(\eta,\xi)
\end{align}
for any $w\in C^{2,\alpha}(\sss^1\times \sss^1)$ and $(\eta,\xi)\in \sss^1\times \sss^1$, where $\varphi^\lambda$ is the solution to \eqref{Eq_HarmonicExtension_KS_OLD}.
\end{proposition}

Now assume that $w(\eta,\xi)=A(\eta)B(\xi)$ is the product of two eigenfunctions $A$ and $B$ of $-\Delta^{\sss^1}$, with associated eigenvalues $\e^2$ and $\delta^2$. As already seen, the solution $\varphi^\lambda$ to \eqref{Eq_HarmonicExtension_KS_OLD} takes the form $\varphi^\lambda (t,\eta,\xi)=l^{\epsilon,\delta,\lambda}(t) w(\eta,\xi)$. In particular, denoting $l:=l^{\epsilon,\delta,\lambda}$, thanks to \eqref{Eq_Prop_Linearization_KS_OLD} it follows that
\begin{align}
\mathbb{L}_\lambda[w](\eta,\xi)=&-\pder{\widetilde{v}}{\theta}(\lambda)\frac{1}{\lambda} \pder{\varphi^\lambda}{t}(1,\eta,\xi) +\pder{^2\widetilde{v}}{\theta^2}(\lambda) w(\eta, \xi)\\
=& - \pder{\widetilde{v}}{\theta}(\lambda)\frac{1}{\lambda} l'(1) w(\eta,\xi)+\pder{^2 \widetilde{v}}{\theta^2}(\lambda) w(\eta, \xi)\\
=&\left( \frac{\tan(\lambda)}{2\lambda} l'(1)-\frac{1}{2\cos^2(\lambda)}\right)w(\eta,\xi).
\end{align}
Hence, $w$ is an eigenfunction of $\mathbb{L}$ with associated eigenvalue
\begin{align}\label{Eq_sigma}
\sigma:=\frac{\tan(\lambda)}{2\lambda} l'(1)-\frac{1}{2} \frac{1}{\cos^2(\lambda)}.
\end{align}
We stress that $\sigma$ depends on $\e, \delta$ and $\lambda$.



\section{Properties of the eigenvalues}\label{Sec:EigenvaluesProperties}

We proceed with the study of the eigenvalue $\sigma$ in the specific case in which $w=w(\eta,\xi)$ is either one of the following
\begin{align}
\cos(n\xi), \quad \quad &\sin(n\xi),\\
\cos(n\eta), \quad \quad &\sin(n\eta),
\end{align}
where $n\in \mathbb{N}\cup \{0\}$. As already mentioned, in this particular cases the function $w$ is an eigenfunction of the operator $\mathbb{L}$. 

\subsection{First case: $w=w(\xi)$} 
The first family of domains we present is obtained starting from $w=\cos(n\xi)$ or $w=\sin(n\xi)$.

\subsubsection{Study of the function $l$}\label{Subsec:L*} 
 Let $\psi=l(t)w(\xi)$ be the solution to
\begin{align}\label{Eq_HarmonicExtension}
\begin{cases}
-\Delta^{g^\lambda} \psi=0 & \inn \Omega\\ \psi=w & \onn \partial \Omega.
\end{cases}
\end{align}
In particular, $l$ satisfies
\begin{align}
\begin{cases}
\lambda^{-2} l''(t) + \lambda^{-1} \left[\frac{\cos(t\lambda)}{\sin(t\lambda)}-\frac{\sin(t\lambda)}{\cos(t\lambda)} \right] l'(t) - \frac{n^2}{\cos^2(t\lambda)} l(t)=0 & \inn (0,1)\\
l(1)=1.
\end{cases}
\end{align}
Set $\theta=t\lambda$ and $L(\theta):=l\left(\frac{\theta}{\lambda}\right)$. Then
\begin{align}
L'(\theta)=\frac{1}{\lambda} l'\left(\frac{\theta}{\lambda} \right) \quad \andd \quad L''(\theta)=\frac{1}{\lambda^2}l''\left(\frac{\theta}{\lambda} \right)
\end{align}
and hence
\begin{align}\label{Eq:L_ball}
\begin{cases}
L''(\theta)+\left[ \frac{\cos(\theta)}{\sin(\theta)}-\frac{\sin(\theta)}{\cos(\theta)}\right] L'(\theta) - \frac{n^2}{\cos^2(\theta)} L(\theta)=0 & \inn (0,\lambda)\\
L(\lambda)=1.
\end{cases}
\end{align}
Notice that the differential equation of  \eqref{Eq:L_ball}  can be written as
\begin{align}\label{Eq:L_ball_short}
\left(\sin(2\theta) L'(\theta)\right)'=\sin(2\theta)\frac{n^2}{\cos^2(\theta)}L(\theta).
\end{align}
In what follows we are going to study the solution of the differential equation of \eqref{Eq:L_ball} in the whole $\left(0,\frac{\pi}{2}\right)$ (and not only in $(0,\lambda)$). Let us fix the solution $L^*$ to
\begin{align}\label{Eq:L^*_ball}
\begin{cases}
(L^*)''(\theta)+\left[ \frac{\cos(\theta)}{\sin(\theta)}-\frac{\sin(\theta)}{\cos(\theta)}\right] (L^*)'(\theta) - \frac{n^2}{\cos^2(\theta)} L^*(\theta)=0 & \inn \left(0,\frac{\pi}{2}\right)\\
L^*\left(\frac{\pi}{4}\right)=1\\
L^*\ \textnormal{regular at } \theta=0,
\end{cases}
\end{align}
see \cite[Theorem 2, Section 7.3]{AD12}. Denote
\begin{align}
&B(\theta)=\theta \left[\frac{\cos(\theta)}{\sin(\theta)}-\frac{\sin(\theta)}{\cos(\theta)} \right] \quad \quad \andd \quad \quad C(\theta)=-\theta^2 \frac{n^2}{\cos^2(\theta)}
\end{align}
and observe that
\begin{align}
B(0)=1, \quad C(0)=0 \quad \andd \quad L''(\theta)+\frac{B(\theta)}{\theta}L'(\theta) + \frac{C(\theta)}{\theta^2}L(\theta)=0.
\end{align}
The associated indicial equation is
\begin{align}
0=r(r-1)+B(0)r+C(0)=r(r-1)+r=r^2
\end{align}
and its root is $r=0$ with multiplicity 2. By Theorem 2 in \cite[Section 7.3]{AD12}, we have two independent Frobenius solutions near $\theta=0$:
\begin{align}
& L_1(\theta)=\sum_{j\geq 0} c_j \theta^j\\
&L_2(\theta)=L_1(\theta)\ln(\theta)+\sum_{j\geq 0} C_j \theta^j
\end{align}
with $c_0\neq 0$ and $C_0\neq 0$. Since we require $L^*$ to be regular at $\theta=0$, it follows that
\begin{align}\label{Eq_L=L_1}
L^*(\theta)=C^* L_1(\theta),
\end{align}
where $C^*$ is a constant with $C^*\not=0$. In particular, $L^*$ is analytic with $L^*(0)=C^*c_0\neq 0$. By continuation $L^*$ exists on $\left(0, \frac{\pi}{2} \right)$.


\begin{proposition}[Properties of $L^*$]\label{Prop:Properties_L_ball}
Let $n \in \mathbb N$. The following properties hold
\begin{enumerate}
\item $(L^*)'(0)=0$;
\item $L^*>0$ in $\left[0,\frac{\pi}{2} \right]$ and $(L^*)''(0)>0$;
\item $(L^*)'>0$ in $\left( 0,\frac{\pi}{2}\right)$.
\end{enumerate}
\end{proposition}
\begin{proof}
Multiplying the differential equation of \eqref{Eq:L^*_ball} by $\theta$ and letting $\theta \to 0^+$ yield that $(L^*)'(0)=0$.
Hence by letting  $\theta \to 0^+$ in the differential equation of  \eqref{Eq:L^*_ball} we have 
\begin{align}\label{the second derivative at 0}
(L^*)^{\prime\prime}(0)=\frac 12n^2L^*(0).
\end{align}

Next, define the set 
\begin{align}
\mathcal{L}^+:=\left\{\theta \in \left(0,\frac{\pi}{2}\right)\ :\ L^*(\theta)>0\ \ \andd\ \ (L^*)'(\theta)>0\right\}.
\end{align}
We claim that $\mathcal{L}^+=\left(0,\frac{\pi}{2}\right)$. Indeed, since $L^*\left(\frac{\pi}{4} \right)=1$, we know that
\begin{itemize}
\item $L(\theta)\equiv 1$ is a supersolution to the differential equation of \eqref{Eq:L^*_ball} in $\left[0,\frac{\pi}{4} \right]$ with boundary condition $L'(0)=0$ and satisfying $L\left(\frac{\pi}{4}\right)=1$;

\item $L(\theta)\equiv 0$ is a solution to the differential equation of \eqref{Eq:L^*_ball} in $\left[0,\frac{\pi}{4} \right]$ with boundary condition $L'(0)=0$ and satisfying $L\left(\frac{\pi}{4}\right)=0$.
\end{itemize}
Since  $L^*(0)\neq 0$, by the comparison principle with the aid of \eqref{the second derivative at 0}, we obtain
\begin{align}\label{Eq:0<L<=1}
0<L^*(\theta)< 1 \quad \inn \left[0,\frac{\pi}{4}\right),
\end{align}
and $(L^*)''(0)>0$ from \eqref{the second derivative at 0}.
By \eqref{Eq:L_ball_short}, the function $\sin(2\theta)(L^*)'$ is increasing in $\left(0,\frac{\pi}{4}\right)$ and hence $(L^*)'>0$ in $\left(0,\frac{\pi}{4}\right)$, implying that $\left(0,\frac{\pi}{4}\right)\subset \mathcal{L}^+$. Now let $(0,b)$ be the connected component of $\mathcal{L}^+$ containing $\left(0,\frac{\pi}{4}\right)$. Clearly, $L^*(b)>0$. If $b<\frac{\pi}{2}$, then $(L^*)'(b)=0$ and so $(L^*)''(b)\leq 0$: this contradicts the differential equation of  \eqref{Eq:L^*_ball}. It follows that $b=\frac{\pi}{2}$ and hence $\mathcal{L}^+=\left(0,\frac{\pi}{2}\right)$. 
\end{proof}


\begin{proposition}[Asymptotic of $L^*$ near $\frac{\pi}{2}$]\label{Prop:Behav_L*_ball}
Let $n \in \mathbb N$.  Set $s=\frac{\pi}{2}-\theta>0$ and $\widetilde{L}^*(s)=L^*(\theta)=L^*\left(\frac{\pi}{2}-s\right)$. Then, $\widetilde{L}^*$ satisfies
\begin{align}
\left(\widetilde{L}^*\right)''(s)-\left[\frac{\sin(s)}{\cos(s)}-\frac{\cos(s)}{\sin(s)}\right] \left(\widetilde{L}^*\right)'(s)-\frac{n^2}{\sin^2(s)}\widetilde{L}^*(s)=0
\end{align}
and, as $s\to 0^+$,
\begin{align}
\widetilde{L}^*(s)\sim c_* s^{-n} \quad and \quad \left(\widetilde{L}^*\right)'(s) \sim -nc_* s^{-n-1}
\end{align}
for some $c_*>0$.
\end{proposition}
\begin{proof}
Let
\begin{align}
B(s)=-s\left[\frac{\sin(s)}{\cos(s)}-\frac{\cos(s)}{\sin(s)}\right] \quad \quad \andd \quad \quad C(s)=-s^2 \frac{n^2}{\sin^2(s)}.
\end{align}
Then, $B(0)=1$ and $C(0)=-n^2$. The indicial equation is
\begin{align}
0=r(r-1)+B(0)r+C(0)=(r-n)(r+n)
\end{align}
whose roots are $n$ and $-n$. By Theorem 2 in \cite[Section 7.3]{AD12}, we have two independent Frobenius solutions near $s=0$:
\begin{align}
&L_1(s)=s^n \sum_{j\geq 0} a_j s^j\\
&L_2(s)= \e L_1(s) \ln(s) + s^{-n} \sum_{j\geq 0} A_j s^j
\end{align}
where $a_0\neq 0$, $A_0\neq 0$ and $\e\in \{0,1\}$. Hence, near $s=0$,
\begin{align}
\widetilde{L}^*(s)=aL_1(s)+bL_2(s)
\end{align}
for some constants $a$ and $b$.

In view of Proposition \ref{Prop:Properties_L_ball}, $b\neq 0$. Indeed, if $b=0$, then
\begin{align}
0=\lim_{s\to 0^+} \widetilde{L}^*(s)=\lim_{\theta \to \frac{\pi}{2}} L^*(\theta)
\end{align}
obtaining a contradiction. Thus, the claim follows.
\end{proof}


\subsubsection{Behaviour of $\sigma_n$}
By Proposition \ref{Prop_Linearization_KS_OLD}
\begin{align}
D_\phi \Big|_{\phi \equiv \lambda} H(\phi)[w] &=\left(\frac{1}{2\lambda} \tan(\lambda) l'(1) - \frac{1}{2\cos^2(\lambda)} \right)w\\
&=: \sigma_n(\lambda) w.
\end{align}
Set $L(\theta)=\frac{L^*(\theta)}{L^*(\lambda)}$, so to have $L(\lambda)=1$ and $L(\theta)=l\left(\frac{\theta}{\lambda} \right)$. Hence
\begin{align}
L'(\theta)=\frac{1}{\lambda} l'\left(\frac{\theta}{\lambda}\right) \quad \andd \quad L'(\lambda)=\frac{1}{\lambda} l'(1).
\end{align}
Thus
\begin{align}
\sigma_n(\lambda)= \frac{1}{2} \tan(\lambda) \frac{(L^*)'(\lambda)}{L^*(\lambda)}-\frac{1}{2\cos^2(\lambda)}.
\end{align}
By Proposition \ref{Prop:Behav_L*_ball}, as $\lambda \to \frac{\pi}{2}$,
\begin{align}
\frac{(L^*)'(\lambda)}{L^*(\lambda)} &=\frac{-(\widetilde{L}^*)'\left(\frac{\pi}{2}-\lambda\right)}{\widetilde{L}^* \left( \frac{\pi}{2}-\lambda\right)}\\
&\sim \frac{n c_*\left(\frac{\pi}{2}-\lambda\right)^{-n-1}}{c_* \left(\frac{\pi}{2}-\lambda\right)^{-n}}\\
&=\frac{n}{\frac{\pi}{2}-\lambda}
\end{align}
and hence we conclude that if $n\ge2$ then
\begin{align}\label{Eq:Asymp_eigenv_infinity_ball}
\sigma_n(\lambda)\sim \frac{n-1}{2\cos^2(\lambda)}\to +\infty \quad \quad \textnormal{as}\ \lambda\to \frac{\pi}{2}.
\end{align}

On the other hand, consider the function $h(\theta):=\frac{\tan(\theta)}{\tan(\lambda)}$, which satisfies
\begin{align}
\left( \sin(2\theta) h'(\theta)\right)'=\frac{2}{\tan(\lambda)} \frac{1}{\cos^2(\theta)}.
\end{align}
We have that
\begin{align}
\sin(2\theta)\frac{n^2}{\cos^2(\theta)}h(\theta) = \frac{2}{\tan(\lambda)} n^2 \left(\frac{1}{\cos^2(\theta)}-1 \right)
\end{align}
and hence
\begin{align}
\left(\sin(2\theta) h'(\theta) \right)'- \sin(2\theta) \frac{n^2}{\cos^2(\theta)} h(\theta) &=\frac{2}{\tan(\lambda)} \left[\frac{1}{\cos^2(\theta)}(1-n^2) + n^2 \right]\\
& \geq \frac{2}{\tan(\lambda)} \left[\frac{1}{\cos^2(\lambda)}(1-n^2)+n^2 \right]\\
& = \frac{2}{\tan(\lambda)} \left[\frac{1-n^2\sin^2(\lambda)}{\cos^2(\lambda)}\right]
\end{align}
in $(0,\lambda)$. Then, it follows that, if $\lambda\le \arcsin\left(\frac1n\right)$ and $n \in \mathbb{N}$, then
\begin{align}
\left(\sin(2\theta) h'(\theta) \right)'- \sin(2\theta) \frac{n^2}{\cos^2(\theta)} h(\theta) \geq 0.
\end{align}
Hence  for such $\lambda$ the function $h$ is a subsolution to the differential equation of \eqref{Eq:L_ball} with boundary data
\begin{align}
h(0)=0 \quad \andd \quad h(\lambda)=1.
\end{align}
Thus, by comparison,
\begin{align}\label{Comparison_L_h}
L'(\lambda)< h'(\lambda)=\frac{1}{\tan(\lambda)} \frac{1}{\cos^2(\lambda)}
\end{align}
implying that
\begin{align}
\sigma_n(\lambda) < \frac{1}{2} \tan(\lambda) \frac{1}{\tan(\lambda) \cos^2(\lambda)}-\frac{1}{2 \cos^2(\lambda)}=0.
\end{align}
Namely, we notice that 
\begin{align}\label{a key inequality of the eigenvalue-1}
\mbox{ if }\ 0 < \lambda \le \arcsin\left(\frac 1n\right) \mbox{ and } n \in \mathbb{N}, \mbox{ then } \ \sigma_n(\lambda) < 0.
\end{align}


\begin{proposition}\label{Prop_Asymptotics_Sigma_1}
The following properties hold:
\begin{enumerate}
\item For every $\lambda \in \left(0,\frac{\pi}{2} \right)$
\begin{align}
0<\liminf_{n\to +\infty} \frac{\sigma_n(\lambda)}{n}\leq \limsup_{n\to +\infty} \frac{\sigma_n(\lambda)}{n}<+\infty.
\end{align}

\item For every $\lambda\in \left(0,\frac{\pi}{2} \right)$ and $i, j\in \mathbb{N}\cup\{0\}$, $i<j$, it holds $\sigma_i(\lambda)<\sigma_j(\lambda)$.

\item The eigenvalues curves $\lambda\mapsto \sigma_n(\lambda)$ have the following asymptotics:
	\begin{enumerate}
		\item $\lim_{\lambda\to 0^+} \sigma_n(\lambda)< 0$ for every $n\geq 1$;
		\item $\lim_{\lambda \to \frac{\pi}{2}^-} \sigma_n(\lambda)=+\infty$ for every $n\geq 2$.
	\end{enumerate}
	
\item For every $n\geq 2$ there exists a unique $\lambda_n \in \left( 0, \frac{\pi}{2}\right)$ such that $\sigma_n(\lambda_n)=0$. Moreover
	\begin{enumerate}
		\item $\sigma'_n(\lambda_n)>0$ for every $n\geq 2$;
		\item $\lambda_n<\lambda_k$ for every $n>k\geq 2$;
		\item $\lambda_n \xrightarrow[]{n\to +\infty}0$.
	\end{enumerate}

\end{enumerate}
\end{proposition}
\begin{proof}The proof is an adaptation of \cite[Proposition 3.1 and Lemma 3.3]{FMW}.
\begin{enumerate}
\item Consider the function $L^*:\left(0,\frac{\pi}{2}\right)\to \rr$ introduced in Subsection \ref{Subsec:L*}. 
As already observed, $l(t)=\frac{L^*(t\lambda)}{L^*(\lambda)}$. Let us define
\begin{align}
f_n(\lambda)=\frac{l'(1)}{\lambda}= \frac{(L^*)'(\lambda)}{L^*(\lambda)}
\end{align}
which satisfies
\begin{align}
f_n'(\lambda) &=\left[ \tan(\lambda)-\cot(\lambda)\right] f_n(\lambda) - f_n^2(\lambda) + \frac{n^2}{\cos^2(\lambda)}
\end{align}
with initial data $f_n(0)=0$. Fix $n\geq 1$ and set
\begin{align}
F_n(\lambda):=f_n(\lambda)-\frac{3n}{\cos(\lambda)},
\end{align}
which satisfies $F_n(0)=-3n<0$, and defined
\begin{align}
M:=\left\{\lambda \in \left(0,\frac{\pi}{2}\right)\ :\ F_n(\lambda)\geq 0\right\}.
\end{align}
Suppose by contradiction that $M\neq \emptyset$ and denote $\lambda^*:=\min M$. Since $F_n(0)<0$, it follows that $F_n(\lambda^*)=0$ and $F'_n(\lambda^*)\geq 0$. In particular $f_n(\lambda^*)=\frac{3n}{\cos(\lambda^*)}$. Then
\begin{align}
& 0 \leq F'_n(\lambda^*)\\
& = [ \tan(\lambda^*)-\cot(\lambda^*)] f_n(\lambda^*) - f_n^2(\lambda^*) + \frac{n^2}{\cos^2(\lambda^*)} \underbrace{- \frac{3n \sin(\lambda^*)}{\cos^2(\lambda^*)}}_{\leq 0}\\
&\leq [ \tan(\lambda^*) \underbrace{-\cot(\lambda^*)}_{\leq 0}] \frac{3n}{\cos(\lambda^*)} - \frac{9n^2}{\cos^2(\lambda^*)} + \frac{n^2}{\cos^2(\lambda^*)}\\
&\leq \frac{1}{\cos(\lambda^*)}\frac{3n}{\cos(\lambda^*)} - \frac{9n^2}{\cos^2(\lambda^*)} + \frac{n^2}{\cos^2(\lambda^*)}\\
&=\frac{n(3-8n)}{\cos^2(\lambda^*)}<0
\end{align}
obtaining a contradiction. Hence $M=\emptyset$, implying $\frac{l'(1)}{\lambda}<\frac{3n}{\cos(\lambda)}$ and so
\begin{align}\label{App:UpperSigma}
\frac{\sigma_n(\lambda)}{n} \leq \frac{\frac{3n}{2\cos(\lambda)} \tan(\lambda) - \frac{1}{2\cos^2(\lambda)}}{n}\le\frac{3 \tan(\lambda)}{2\cos(\lambda)} <+\infty.
\end{align}

For the other estimate define $g_n(\lambda)=n \tan(\lambda)$, which satisfies
\begin{align}
& g_n'(\lambda)-[\tan(\lambda)-\cot(\lambda)]g_n+g_n^2-\frac{n^2}{\cos^2(\lambda)}\\
&= 2n-n^2 \le 0
\end{align}
for $n \ge 2$ and $g_n(0)=0$. Hence, the function $h_n:=f_n-g_n$ satisfies $h_n(0)=0$ and
\begin{align}
h_n'(\lambda) \geq [\tan(\lambda)-\cot(\lambda)-f_n(\lambda)-g_n(\lambda)] h_n(\lambda)
\end{align}
implying, by Gronwall's inequality, that $h_n\geq 0$ and hence $f_n(\lambda)\geq n \tan(\lambda)$. As before, this implies
\begin{align}\label{App:LowerSigma}
\frac{\sigma_n(\lambda)}{n} \geq \frac{\frac{n}{2} \tan^2(\lambda) - \frac{1}{2\cos^2(\lambda)}}{n}\xrightarrow[]{n\to + \infty} \frac{\tan^2(\lambda)}{2}.
\end{align}

\item For $i<j$ define $h:=f_j-f_i$ which satisfies
\begin{align}
\begin{cases}
h'(\lambda)= [\tan(\lambda)-\cot(\lambda)-f_j(\lambda)-f_i(\lambda)]h(\lambda) + \frac{j^2-i^2}{\cos^2(\lambda)}\\
h(0)=0.
\end{cases}
\end{align}
Since $\frac{j^2-i^2}{\cos^2(\lambda)}>0$, it follows, again by Gronwall's inequality, $h>0$ and hence $f_j>f_i$. This implies $\sigma_j(\lambda)>\sigma_i(\lambda)$.

\item Both $(a)$ and $(b)$ follow from what seen in the present subsection.

\item The existence of $\lambda_n$ follows by point $(3)$. For the uniqueness, if $\lambda_n\in \left(0,\frac{\pi}{2}\right)$ is such that $\sigma_n(\lambda_n)=0$, then
\begin{align}\label{Eq:App_L'(lambda_0)}
L'(\lambda_n)=\frac{1}{\sin(\lambda_n)\cos(\lambda_n)}.
\end{align}
Using that $L'(\lambda)=\frac{(L^*)'(\lambda)}{L^*(\lambda)}$ and \eqref{Eq:L^*_ball}, one gets
\begin{align}\label{sigma'_lambda_n}
\sigma'_n(\lambda_n)=\frac{1}{2\cos^3(\lambda_n)\sin(\lambda_n)}\left(n^2 \sin^2(\lambda_n)-1 \right).
\end{align}
Then, since $\sigma_n(\lambda_n)=0$, the fact that $\sigma_n'(\lambda_n)>0$ follows from  \eqref{a key inequality of the eigenvalue-1}.
This proves both (a) and the uniqueness of $\lambda_n$. Point (b) follows from (2). For proving (c) we only have to notice that by \eqref{App:LowerSigma}
\begin{align}
0=\sigma_n(\lambda_n)\geq \frac{1}{2 \cos^2(\lambda_n)} \left(n \sin^2(\lambda_n)-1\right)
\end{align}
which implies $\lambda_n \in \left(0,\arcsin\left(\frac{1}{\sqrt{n}}\right)\right)$.
\end{enumerate}
\end{proof}


\subsection{Second case: $w=w(\eta)$} One can proceed similarly to the previous case, obtaining

\begin{proposition}\label{Prop_Asymptotics_Sigma_1_Second_Case}
The following properties hold:
\begin{enumerate}
\item For every $\lambda \in \left(0,\frac{\pi}{2} \right)$
\begin{align}
0<\liminf_{n\to +\infty} \frac{\sigma_n(\lambda)}{n}\leq \limsup_{n\to +\infty} \frac{\sigma_n(\lambda)}{n}<+\infty.
\end{align}

\item For every $\lambda\in \left(0,\frac{\pi}{2} \right)$ and $i,j\in \mathbb{N}\cup\{0\}$, $i<j$, it holds $\sigma_i(\lambda)<\sigma_j(\lambda)$.

\item The eigenvalues curves $\lambda\mapsto \sigma_n(\lambda)$ have the following asymptotics:
	\begin{enumerate}
		\item $\lim_{\lambda\to 0^+} \sigma_n(\lambda)> 0$ for every $n\geq 2$;
		\item $\lim_{\lambda \to \frac{\pi}{2}^-} \sigma_n(\lambda)=-\infty$ for every $n\geq 1$.
	\end{enumerate}
	
\item For every $n\geq 2$ there exists a unique $\lambda_n \in \left( 0, \frac{\pi}{2}\right)$ such that $\sigma_n(\lambda_n)=0$. Moreover
	\begin{enumerate}
		\item $\sigma'_n(\lambda_n)<0$ for every $n\geq 2$;
		\item $\lambda_n>\lambda_k$ for every $n>k\geq 2$;
		\item $\lambda_n \xrightarrow[]{n\to +\infty}\frac{\pi}{2}$.
	\end{enumerate}

\end{enumerate}
\end{proposition}

\noindent For the complete proof see Appendix \ref{Sec:AppB}.



\section{Bifurcation argument}\label{Sec:Bifurcation}
We introduce two spaces
\begin{align}
S_1:=\{ \phi \in C^{1,\alpha}(\sss^1_\xi )\, :\, \phi \mbox{ is even in }\xi \}\ \mbox{ and }\  S_2:=\{ \phi \in C^{2,\alpha}(\sss^1_\xi )\, :\, \phi \mbox{ is even in }\xi \}.
\end{align}
We denote
\begin{itemize}
\item $\mathcal{F}:=\{\cos(m\xi)\}_{m\geq 0}$ and hence we have
\begin{equation}
S_1=\overline{\textnormal{Span}_{\rr}\mathcal{F}}^{||\cdot||_{C^{1,\alpha}}}\ \mbox{ and }\ 
S_2=\overline{\textnormal{Span}_{\rr}\mathcal{F}}^{||\cdot||_{C^{2,\alpha}}};
\end{equation}

\item for every $k\in \mathbb{N}\cup\{0\}$
\begin{align}
\mathcal{H}_k:=\overline{\textnormal{Span}_{\rr}\mathcal{F}}^{||\cdot||_{H^k}};
\end{align}

\item $\mathcal{L}^2:=\mathcal{H}_0$ with inner scalar product
\begin{align}
\scal{u,v}:=\int_{\sss^1\times \sss^1} uv \ \textnormal{da};
\end{align}

\item for every $j\in \mathbb{N}\cup\{0\}$
\begin{align}
V^j:=\textnormal{Span}_\rr \{\cos(j\xi)\}\subset \cap_k \mathcal{H}_k;
\end{align}

\item for every $j\in \mathbb{N}\cup\{0\}$
\begin{align}
P_j:\mathcal{L}^2\to \mathcal{L}^2 \quad \textnormal{the } \mathcal{L}^2\textnormal{-orthogonal projection on } V^j;
\end{align}

\item for every $j,k\in \mathbb{N}\cup\{0\}$
\begin{align}
W_k^j:=(V^j)^{\bot_{\mathcal{L}^2|_{\mathcal{H}_k}}}:=\{v\in \mathcal{H}_k\ :\ P_j v=0\},
\end{align}
the $\mathcal{L}^2$-orthogonal complement of $V^j$ in $\mathcal{H}_k$.
\end{itemize}


\begin{remark}
In the present section, we are going to prove Theorem \ref{Thm:Main} for the case $w=w(\xi)$, i.e. constructing a family of non-isoparametric Serrin domains with small volumes (i.e. $\lambda_j\xrightarrow[]{j\to +\infty} 0$). For the case $w=w(\eta)$, which provides a family of non-isoparametric Serrin domains with volume close to $|\sss^3|$ (i.e. $\lambda_j\xrightarrow[]{j\to +\infty} \frac{\pi}{2}$), one just has to retrace, step by step, what we are going to do by replacing previous definitions with the following ones
\begin{align}
\mathcal{F}:=\{\cos(m\eta)\}_{m\geq 0} \quad \quad \andd \quad \quad V^j:=\textnormal{Span}_\rr \{\cos(j\eta)\},
\end{align}
hence modifying $S_1,S_2,\mathcal{H}_k,\mathcal{L}^2,P_j$ and $W^j_k$ in accordance with these one, and using Proposition \ref{Prop_Asymptotics_Sigma_1_Second_Case} instead of Proposition \ref{Prop_Asymptotics_Sigma_1}.
\end{remark}


\subsection{Functional properties}
\noindent The next proposition is needed to ensure that the operator $\mathbb{L}_\lambda$ can be extended to the functional space $\mathcal{H}_2$. It is obtained as in \cite[Proposition 4.1]{FMW}.


\begin{proposition}\label{Prop_Prop4.1_m=n}
For any fixed $\lambda \in \left(0, \frac{\pi}{2}\right)$, the linear map
\begin{align}\label{LS_Prop_Prop4.1_m=n}
\mathbb{L}_\lambda^{S_2}:=\mathbb{L}_\lambda\Big|_{S_2}:S_2\to S_1,
\end{align}
where $\mathbb{L}_\lambda$ is defined as \eqref{Eq_Prop_Linearization_KS_OLD}, extends to a continuous linear map
\begin{align}\label{Extension_Prop_Prop4.1_m=n}
\mathbb{L}^{S_2}_\lambda:\mathcal{H}_2 & \to \mathcal{H}_1\\
w &\mapsto \sum_{j\in \mathbb{N}\cup\{0\}} \sigma_j(\lambda) P_j w.
\end{align}
Moreover, for any $j\in \mathbb{N}\cup\{0\}$ the following operator is an isomorphism
\begin{align}\label{Isom_Prop_Prop4.1_m=n}
\mathbb{L}^{S_2}_\lambda-\sigma_j(\lambda):W_2^j\to W_1^j.
\end{align}
\end{proposition}
\begin{proof}
Observe that $\mathcal{H}_k$ can be characterized as the subspace of functions $v\in \mathcal{L}^2$ such that
\begin{align}
\sum_{j\in \mathbb{N}\cup\{0\}} (1+j^2)^k \scal{P_j v, P_j v}<+\infty.
\end{align}
Hence, fixed $w\in \mathcal{H}_2$, one has
\begin{align}
\sum_{j\in \mathbb{N}\cup\{0\}} (1+j^2)^2 \scal{P_j w, P_j w}<+\infty
\end{align}
and so, by $(1)$ of Proposition \ref{Prop_Asymptotics_Sigma_1},
\begin{align}
\sum_{j\in \mathbb{N}\cup\{0\}} (1+j^2) & \scal{P_j \mathbb{L}^{S_2}_\lambda w, P_j \mathbb{L}^{S_2}_\lambda w}\\
&=\sum_{j\in \mathbb{N}\cup\{0\}} (1+j^2) \sigma_j^2(\lambda)\scal{P_j w, P_j w }\\
&\le \sum_{j\in \mathbb{N}\cup\{0\}} (1+j^2) E_1^2(1+ j^2) \scal{P_j w, P_j w }\\
&\le \sum_{j\in \mathbb{N}\cup\{0\}} (1+j^2)^2 E_1^2 \scal{P_j w, P_j w }<+\infty
\end{align}
where $E_1$ is a positive constant,
implying $\mathbb{L}^{S_2}_\lambda w \in \mathcal{H}_1$. Hence, \eqref{Extension_Prop_Prop4.1_m=n} defines a continuous linear operator. Since on any finite linear combination of elements of $\mathcal{F}$ this operator coincides with \eqref{LS_Prop_Prop4.1_m=n} and since $\mathcal{F}$ is dense dense both in $S_2$ and $\mathcal{H}_2$, the claimed extension follows by continuity.

Clearly, $\mathbb{L}^{S_2}_\lambda$ maps $W_2^j$ into $W_1^j$. Moreover the lower estimate in $(1)$ of Proposition \ref{Prop_Asymptotics_Sigma_1} implies that the operator
\begin{align}\label{App1_Prop_Prop4.1_m=n}
R_j^\lambda: W_1^j & \to W_2^j\\
w & \mapsto \sum_{m\neq j} \frac{1}{\sigma_m(\lambda)-\sigma_j(\lambda)} P_m w
\end{align}
is well defined and continuous. Indeed, for any $w\in W_1^j$
\begin{align}
P_j \sum_{m\neq j} \frac{1}{\sigma_m(\lambda)-\sigma_j(\lambda)} P_m w &=\sum_{m\neq j} \frac{1}{\sigma_m(\lambda)-\sigma_j(\lambda)} P_j P_m w=0
\end{align}
and, by the lower estimate in $(1)$ of Proposition \ref{Prop_Asymptotics_Sigma_1},
\begin{align}
\sum_{t\in \mathbb{N}\cup\{0\}} (1+t^2)^2 &\scal{P_t \sum_{m\neq j} \frac{1}{\sigma_m(\lambda)-\sigma_j(\lambda)} P_m w, P_t \sum_{m\neq j} \frac{1}{\sigma_m(\lambda)-\sigma_j(\lambda)} P_m w}\\
&=\sum_{t\neq j} (1+t^2)^2 \left(\frac{1}{\sigma_t(\lambda)-\sigma_j(\lambda)}\right)^2 \scal{P_t w, P_t w}\\
&\le\sum_{t\neq j} (1+t^2)^2 \frac{1}{E_2^2(1+t^2)} \scal{P_t w, P_t w}\\
&=\frac 1{E_2^2}\sum_{t\neq j} (1+t^2) \scal{P_t w, P_t w} <+\infty, 
\end{align}
since $w\in \mathcal{H}_1$, where $E_2$ is a positive constant. By a direct computation, the inverse of $R_j^\lambda$ is $\mathbb{L}^{S_2}_\lambda-\sigma_j(\lambda)$: for $w\in W_2^j$
\begin{align}
R_j^\lambda & ((\mathbb{L}^{S_2}_\lambda-\sigma_j(\lambda))(w))=\sum_{m\neq j} \frac{1}{\sigma_m(\lambda)-\sigma_j(\lambda)} P_m ((\mathbb{L}^{S_2}_\lambda-\sigma_j(\lambda))(w))\\
&=\sum_{m\neq j} \frac{1}{\sigma_m(\lambda)-\sigma_j(\lambda)} P_m \sum_{t\in \mathbb{N}\cup\{0\}} \sigma_t(\lambda) P_t w - \sum_{m\neq j} \frac{1}{\sigma_m(\lambda)-\sigma_j(\lambda)} P_m \sigma_j(\lambda) w\\
&=\sum_{m\neq j} \frac{\sigma_m(\lambda)}{\sigma_m(\lambda)-\sigma_j(\lambda)} P_m w - \sum_{m\neq j} \frac{\sigma_j(\lambda)}{\sigma_m(\lambda)-\sigma_j(\lambda)} P_m w\\
&=w;
\end{align}
while for $w\in W_1^j$
\begin{align}
(\mathbb{L}^{S_2}_\lambda- & \sigma_j(\lambda))R_j^\lambda w =(\mathbb{L}^{S_2}_\lambda-\sigma_j(\lambda)) \sum_{m\neq j} \frac{1}{\sigma_m(\lambda)-\sigma_j(\lambda)} P_m w\\
&=\sum_{t\in \mathbb{N}\cup\{0\}} \sigma_t(\lambda) P_t \sum_{m\neq j} \frac{1}{\sigma_m(\lambda)-\sigma_j(\lambda)} P_m w \\
&\qquad- \sigma_j(\lambda)\sum_{m\neq j} \frac{1}{\sigma_m(\lambda)-\sigma_j(\lambda)} P_m w\\
&=\sum_{m\neq j} \frac{\sigma_m(\lambda)}{\sigma_m(\lambda)-\sigma_j(\lambda)} P_m w - \sum_{m\neq j} \frac{\sigma_j(\lambda)}{\sigma_m(\lambda)-\sigma_j(\lambda)} P_m w\\
&=w.
\end{align}
We get the claim.
\end{proof}


\begin{remark}
Let $w\in \mathcal{H}_2$ and let $\varphi^\lambda\in W^{1,2}(\Omega)$ be the unique weak solution to \eqref{Eq_HarmonicExtension}, which only depends on $t$ and $\xi$. Moreover, if $w\in S_2$ the function $\varphi^\lambda$ is even with respect to the variable $\xi$. Hence, for any fixed $t$ we have $\varphi^\lambda(t,\cdot)\in \mathcal{H}_2$. By elliptic regularity, $\varphi^\lambda \in W^{2,2}(\Omega)$ and, as seen in \cite[Remark 4.2]{FMW}, the above extension $\mathbb{L}^{S_2}_\lambda$ can be characterized as the operator
\begin{align}
\mathbb{L}_\lambda [w](\xi)=\frac{\tan(\lambda)}{2\lambda} \pder{\varphi^\lambda}{t}(1,\xi) - \frac{1}{2} \frac{1}{\cos^2(\lambda)} w(\xi)
\end{align}
where $\pder{\varphi^\lambda}{t}$ is considered in the sense of traces.
\end{remark}

\subsection{Proof of Theorem \ref{Thm:Main}}
Now consider the following space
\begin{align}
\mathcal{O}:=\left\{(\lambda,\phi)\in \left(0,\frac{\pi}{2}\right)\times S_2\ :\ -\lambda<\phi< \frac{\pi}{2}-\lambda\right\}.
\end{align}
Let
\begin{align}
G:\mathcal{O}&\to S_1\\
(\lambda, \phi)&\mapsto H(\lambda+\phi)-H(\lambda),
\end{align}
where Lemma \ref{Lem_evenness} guarantees that $G$ is well-defined.
In particular
\begin{align}
G(\lambda,\phi)&=H(\lambda+\phi)-\pder{\widetilde{v}}{\theta}\left( \lambda\right)\\
&=H(\lambda+\phi)+\frac{1}{2}\tan(\lambda)
\end{align}
implying $G(\lambda,0)=0$ and
\begin{align}
\left(D_\phi \Big|_{\phi=0} G\right) (\lambda,0)=\mathbb{L}_\lambda^{S_2}\in \mathcal{L}(S_2,S_1).
\end{align}

\noindent Similarly to \cite[Proposition 5.1]{FMW}, we can prove that

\begin{proposition}\label{Prop_Prop5.1_m=n}
Consider $\lambda_j$, with $j\geq 2$, as given in Proposition \ref{Prop_Asymptotics_Sigma_1}, point $(4)$. For every $j\geq 2$ the linear operator
\begin{align}
\mathbb{L}_j:=\mathbb{L}_{\lambda_j}^{S_2} \in \mathcal{L}(S_{2},S_1)
\end{align}
has the following properties
\begin{enumerate}
\item the kernel $\textnormal{ker}(\mathbb{L}_j)$ of $\mathbb{L}_j$ is spanned by $w_j(\xi)=\cos(j\xi)$;
\item the range $\textnormal{Im}(\mathbb{L}_j)$ of $\mathbb{L}_j$ is given by
\begin{align}
\textnormal{Im}(\mathbb{L}_j)=\left\{v\in S_1\ :\ \scal{v,w_j}=0\right\}.
\end{align}
\end{enumerate}
Moreover,
\begin{align}
\frac{\partial}{\partial \lambda}\Big|_{\lambda=\lambda_j} \mathbb{L}_j w_j \notin \textnormal{Im}(\mathbb{L}_j).
\end{align}
\end{proposition}
\begin{proof}
Let us define
\begin{align}
&S_{1}^j:=\left\{ v\in S_1\ :\ \scal{v,w_j} =0 \right\}\\
&S_{2}^j:=\left\{ v\in S_{2}\ :\ \scal{v,w_j} =0 \right\}.
\end{align}
We want to prove that $\mathbb{L}_j$ defines an isomorphism between $S_{2}^j$ and $S_{1}^j$: this, together with the fact that $\mathbb{L}_j w_j=\sigma_j(\lambda_j)w_j=0$, will provide the claims (1) and (2). By Proposition \ref{Prop_Prop4.1_m=n} we have a continuous linear operator
\begin{align}
\mathbb{L}_j:S_{2}^j\to S_{1}^j.
\end{align}
Now note that
\begin{align}\label{App1_Prop_Prop5.1_m=n}
S_{2}^j=W_2^j\cap S_{2} \quad \andd \quad S_{1}^j=W_1^j\cap S_1.
\end{align}
By Proposition \ref{Prop_Prop4.1_m=n} we have that
\begin{align}\label{App2_Prop_Prop5.1_m=n}
\mathbb{L}_j:W_2^j\to W_1^j \quad \textnormal{defines an isomorphism}
\end{align}
and so, by \eqref{App1_Prop_Prop5.1_m=n},
\begin{align}
\mathbb{L}_j:S_{2}^j\to S_{1}^j \quad \textnormal{is injective.}
\end{align}

For the surjectivity, fix $f\in S_{1}^j$ and denote by $w\in W_2^j$ the unique solution to $\mathbb{L}_j w=f$ (recall \eqref{App2_Prop_Prop5.1_m=n}). We observe that the (extended) operator $\mathbb{L}_j$ can be characterized as the map $w \mapsto \mathbb{L}_j [w]$ so that
\begin{align}
\mathbb{L}_j [w](\eta,\xi)=-\pder{\widetilde{v}}{\theta}(\lambda)\frac{1}{\lambda} \pder{\varphi}{t}(1,\eta,\xi) +\pder{^2\widetilde{v}}{\theta^2}(\lambda) w(\eta, \xi)
\end{align}
for almost every $(\eta,\xi)\in \sss^1\times \sss^1$, where $\varphi\in W^{1,2}(\Omega)$ is the unique weak solution to \eqref{Eq_HarmonicExtension_KS_OLD}. Moreover, if $w\in S_2$, then also $\varphi$ only depends on $\xi$. By standard elliptic regularity, $\varphi \in C^{2,\alpha}(\overline{\Omega})$: by passing to the trace, we conclude that $w\in C^{2,\alpha}(\sss^1\times \sss^1)$ and hence $w\in S_{2}^j$, providing the desired surjectivity.

To conclude, we can observe that
\begin{align}
\partial_\lambda\Big|_{\lambda=\lambda_j} \mathbb{L}_{\lambda_j} w_j&=\partial_\lambda\Big|_{\lambda=\lambda_j} \sigma_j(\lambda) w_j\\
&=\sigma'_j(\lambda_j)w_j
\end{align}
and since $\sigma'_j(\lambda_j)>0$ by Proposition \ref{Prop_Asymptotics_Sigma_1}, point $(4)$, it follows
\begin{align}
\scal{\partial_\lambda\Big|_{\lambda=\lambda_j} \mathbb{L}_{\lambda_j} w_j, w_j}=\sigma_j'(\lambda_j)\scal{w_j, w_j}\neq 0
\end{align}
and hence $\partial_\lambda\Big|_{\lambda=\lambda_j} \mathbb{L}_{\lambda_j} w_j \notin R(\mathbb{L}_j^{S_2})$.
\end{proof}

\begin{proof}[Proof of Theorem \ref{Thm:Main}]
We are in force to apply Theorem \ref{Thm:CRBifurcation} to the operator $G$ acting on $\mathcal{O}$ and with
\begin{align}
\mathcal{C}=W_2^j,\quad \lambda_*=\lambda_j \quad \andd \quad a_*=\cos(j\xi),
\end{align}
hence providing the existence of a continuous curve
\begin{align}
(-\e,\e) & \to (\rr\times \mathcal{C})\cap \mathcal{O}\\
s & \mapsto (\lambda(s),w(s)),
\end{align}
where $(\lambda(0),w(0))=(\lambda_*,0)$ and so that for every $s\in (-\e,\e)$ one has $s(a_*+w(s))\in \mathcal{O}$ and
\begin{align}
G(\lambda(s),s(a_*+w(s)))=0.
\end{align}
This exactly means that for every $s\in (-\e,\e)$
\begin{align}
\pder{u_{\lambda_* + s(a_*+w(s))}}{\nu_{\lambda_* + s(a_*+w(s))}}=\pder{u_{\lambda_*}}{\nu_{\lambda_*}}\equiv \textnormal{const.}
\end{align}
and so that $\Omega_{\lambda_* + s(a_*+w(s))}$ is a Serrin domain. Moreover, we stress that $\partial \Omega_{\lambda_* + s(a_*+w(s))}$ is not isometric to any Clifford torus since $w(s)\in W_2^j$ for every $s\in (-\e,\e)$, implying that $a_*+w(s)$ is nonconstant for every $s \in (-\e,\e)$.
\end{proof}



\bigskip

\noindent\textbf{Acknowledgements.} The first author has been supported by ``Centro di Ricerca Matematica Ennio De Giorgi'' and he is a member of GNAMPA-INdAM. The second author has been supported by JSPS KAKENHI Grant Number JP22K03381. This research was started while the first author visited Tohoku University; he wishes to thank its kind hospitality.


\appendix

\section{Technical tools}\label{Sec:AppA}

Here we presents some technical tools necessary to the study of the spectrum of linearization of $H$. All the results present in this appendix are slightly variation of some results proved in \cite{FMW}.

To this aim we have to compute the Gateaux derivative of the functional
\begin{align}
\Psi:\mathcal{U}&\to C^{2,\alpha}(\overline{\Omega}, \rr\times \sss^1 \times \sss^1)\\
\phi&\mapsto \Psi^\phi,
\end{align}
where $\Psi^\phi$ is defined as in \eqref{Eq_ParamPsi_KS_OLD}, and of the functional
\begin{align}
\mathcal{M}:\mathcal{U}&\to C^{1,\alpha}(\partial \Omega)\\
\phi &\mapsto \frac{\partial m_\phi}{\partial \nu_\phi},
\end{align}
where $m:\mathcal{U}\to C^{2,\alpha}(\overline{\Omega})$, $\phi \mapsto m_\phi$, is a general smooth map. We stress that both $\Psi$ and $\mathcal{M}$ are smooth.

\begin{lemma}\label{Lem_App1_Linearization_KS_OLD}
For any $w\in C^{2,\alpha}(\sss^1\times \sss^1)$ it holds
\begin{align}
D_\phi \Psi^\phi [w] (t,\eta,\xi)=t w(\eta,\xi) \frac{\partial}{\partial t}\Big|_{(t,\eta,\xi)}
\end{align}
and
\begin{align}
D_\phi \mathcal{M}(\phi)[w](\eta,\xi)= \frac{\partial m_\phi}{\partial \widetilde{\nu}_\phi(w)}(\eta,\xi)+ \frac{\partial}{\partial \nu_\phi}\left(D_\phi m_\phi [w] \right)(\eta,\xi),
\end{align}
where
\begin{align}
\widetilde{\nu}_{\phi}(w)=D_\phi \nu_\phi [w].
\end{align}
\end{lemma}
\begin{proof}
In our setting the computation is straightforward. See also \cite[Lemma 2.5]{FMW}.
\end{proof}

Let $ u^\phi := (\Psi^\phi)^* \widetilde{v} = \widetilde{v} \circ \Psi^\phi$: since $\Psi^\phi:(\Omega, g^\phi)\to (\Omega_\phi, g)$ is an isometry (recall: $g^\phi:=(\Psi^\phi)^*g$) and since $-\Delta^g \widetilde{v}=1$, the function $u^\phi$ satisfies
\begin{align}
-\Delta^{g^\phi} u^\phi=1 \quad \inn \Omega.
\end{align}
Define the functional
\begin{align}
h:\mathcal{U}\subset C^{2,\alpha}(\sss^1\times \sss^1) &\to C^{1,\alpha}(\sss^1\times \sss^1)\\
\phi & \mapsto \pder{u^\phi}{\nu_\phi} (1,\cdot, \cdot).
\end{align}

\begin{lemma}\label{Lem_App2_Linearization_KS_OLD}
For any $w\in C^{2,\alpha}(\sss^1\times \sss^1)$, $\lambda\in \left(0,\frac{\pi}{2}\right)$ and $(\eta,\xi)\in \sss^1\times \sss^1$ it holds
\begin{align}
\left(D_\phi\Big|_{\phi\equiv \lambda} h\right) [w](\eta,\xi)=-\frac{1}{2\cos^2(\lambda)} w(\eta,\xi).
\end{align}
\end{lemma}
\begin{proof}
The claim follows as in \cite[Lemma 2.6]{FMW}. Let us define
\begin{align}
&\widehat{\mu_\phi}(\eta,\xi):=\mu_\phi\left(\Psi^\phi(1,\eta,\xi)\right)=\pder{}{\theta} \Bigg|_{\Psi^\phi(1,\eta,\xi)}
\end{align}
and observe that
\begin{align}
h(\phi)(\eta,\xi)&=g^\phi \left(\nu_\phi, \nabla^{g^\phi} u^\phi\right)\Big|_{ (1, \eta, \xi)}\\
&=g\left(\widehat{\mu_\phi}, \nabla^g \widetilde{v} \right)\Big|_{\Psi^\phi(1,\eta, \xi)}\\
&= g\left(\widehat{\mu_\phi}, \pder{\widetilde{v}}{\theta} \frac{\partial}{\partial \theta} \right)\Big|_{\Psi^\phi(1,\eta, \xi)}
\end{align}
where in the second equality we have used the fact that $\Psi^\phi$ is an isometry. Noticing that the map
\begin{align}
\mathcal{U}&\to C^{1,\alpha}(\sss^1\times \sss^1)\\
\phi &\mapsto g\left(\widehat{\mu_\phi}, \widehat{\mu_\phi} \right)\equiv 1
\end{align}
is constant, then it follows that
\begin{align}
g\left(D_\phi \widehat{\mu_\phi} [w], \widehat{\mu_\phi} \right)=\frac{1}{2} \left(D_\phi g\left(\widehat{\mu_\phi}, \widehat{\mu_\phi} \right)\right)[w]=0
\end{align}
for every $w \in C^{2,\alpha}(\sss^1\times \sss^1)$ and every $\phi \in \mathcal{U}$. Hence, for any $\lambda \in (0,\frac{\pi}{2})$, $w\in C^{2,\alpha}(\sss^1\times\sss^1)$ and $(\eta,\xi) \in \sss^1\times \sss^1$
\begin{align}
\left(D_\phi \Big|_{\phi\equiv \lambda} h(\Phi)\right) [w] =& g\left(\left(D_\Phi \Big|_{\phi\equiv \lambda} \widehat{\mu_\phi} \right)[w], \pder{\widetilde{v}}{\theta} \frac{\partial}{\partial \theta} \right)\Big|_{\Psi^\lambda(1,\eta,\xi)}\\
&+g\left(\widehat{\mu_\lambda}, \pder{^2\widetilde{v}}{\theta^2} D_\phi\Big|_{\phi\equiv \lambda} \Psi^\phi [w] \frac{\partial}{\partial \theta} \right)\Big|_{\Psi^\lambda(1,\eta,\xi)}\\
=& \pder{\widetilde{v}}{\theta}(\lambda) \underbrace{g\left(\left(D_\phi \Big|_{\phi\equiv \lambda} \widehat{\mu_\phi}\right)[w], \widehat{\mu_\lambda} \right)\Big|_{\Psi^\lambda( 1,\eta,\xi)}}_{=0}\\
&+g\left(\widehat{\mu_\lambda}, \pder{^2\widetilde{v}}{\theta^2}  D_\phi\Big|_{\phi\equiv \lambda} \Psi^\phi [w] \frac{\partial}{\partial \theta} \right)\Big|_{\Psi^\lambda(1,\eta,\xi)}\\
=&g\left(\widehat{\mu_\lambda}, \widehat{\mu_\lambda} \right)\Big|_{\Psi^\lambda(1,\eta,\xi)} \pder{^2\widetilde{v}}{\theta^2}(\lambda) w(\eta, \xi)\\
=& \pder{^2\widetilde{v}}{\theta^2}(\lambda) w(\eta, \xi).
\end{align}
\end{proof}

\begin{proof}[Proof of Proposition \ref{Prop_Linearization_KS_OLD}] The proof proceeds exactly as in \cite[Proposition 2.4]{FMW}. Fix $\phi \in \mathcal{U}$ and define $a_\phi:=u_\phi - u^\phi$, which is a solution to
\begin{align}\label{Eq_Problem_a-phi_KS_OLD}
\begin{cases}
-\Delta^{g^\phi} a_\phi=0 & \inn \Omega\\
a_\phi=-u^\phi & \onn \partial \Omega.
\end{cases}
\end{align}
In case $\phi\equiv \lambda$ is constant, for any $(t,\eta,\xi)\in \overline{\Omega}$ we have
\begin{align}
u^\lambda(t,\eta,\xi) = \widetilde{v}(t\lambda)
\end{align}
and
\begin{align}
u_\lambda(t,\eta,\xi)=\widetilde{v}(t\lambda)-\widetilde{v}(\lambda)
\end{align}
implying
\begin{align}\label{Eq_a-phi_constant_KS_OLD}
a_\lambda = - \widetilde{v}\left(\lambda\right),
\end{align}
constant on $\overline{\Omega}$. Now define
\begin{align}
T:\mathcal{U}&\to C^{0,\alpha}(\overline{\Omega})\\
\phi&\mapsto \Delta^{g^\phi} a_\phi.
\end{align}
In particular, $T\equiv 0$ on $\mathcal{U}$ and hence, differentiating $T$ with respect to $\phi$ and evaluating at $\phi\equiv \lambda$ constant, for every fixed $w\in C^{2,\alpha}(\overline{\Omega})$ we get
\begin{align}\label{App1_Prop_Linearization_KS_OLD}
0=D_\phi \Big|_{\phi\equiv \lambda} T(\phi) [w]=\Delta^{g^\lambda} \left(D_\phi \Big|_{\phi\equiv \lambda} a_\phi [w] \right) \quad \onn \Omega,
\end{align}
where the second identity follows from a direct computation together with \eqref{Eq_a-phi_constant_KS_OLD}. Moreover, differentiating the boundary condition in \eqref{Eq_Problem_a-phi_KS_OLD} with respect to $\phi$ and evaluating at $\phi\equiv \lambda$, we obtain
\begin{equation}\label{App2_Prop_Linearization_KS_OLD}
\begin{split}
D_\phi \Big|_{\phi\equiv \lambda} a_\phi [w] (1,\eta,\xi)=-\widetilde{v}'(\lambda) w(\eta,\xi)
\end{split}
\end{equation}
thanks to \eqref{Eq_a-phi_constant_KS_OLD}. From \eqref{App1_Prop_Linearization_KS_OLD} and \eqref{App2_Prop_Linearization_KS_OLD}, it follows that
\begin{align}
\varphi^\lambda=-\frac{D_\phi \Big|_{\phi\equiv \lambda} a_\phi [w]}{\widetilde{v}'(\lambda)},
\end{align}
where $\varphi^\lambda$ is the solution to \eqref{Eq_HarmonicExtension_KS_OLD}. Again by Lemma \ref{Lem_App1_Linearization_KS_OLD} it also follows
\begin{align}\label{App4_Prop_Linearization_KS_OLD}
D_\phi \Big|_{\phi\equiv \lambda} \left(\frac{\partial a_\phi}{\partial \nu_\phi}\right)[w](1,\eta,\xi) =& \frac{\partial a_\lambda}{\partial \widetilde{\nu}_\lambda (w)}(1,\eta,\xi)\\
&+ \frac{\partial}{\partial \nu_\lambda}\left(D_\phi |_{\phi\equiv \lambda} a_\phi [w] \right)(1,\eta,\xi)\\
=& \frac{1}{\lambda} \frac{\partial}{\partial t}\Big|_{t=1}\left(D_\phi |_{\phi\equiv \lambda} a_\phi [w] \right)(\eta,\xi)
\end{align}
for any $(\eta,\xi)\in \sss^1\times \sss^1$, where in the second equality we have used \eqref{Eq_a-phi_constant_KS_OLD} and the fact that $\nu_\lambda(t,\eta,\xi)=\frac{1}{\lambda}\frac{\partial}{\partial t}$ on $\partial \Omega$. Recalling that $u_\phi=a_\phi+u^\phi$, by \eqref{App4_Prop_Linearization_KS_OLD} and by Lemma \ref{Lem_App2_Linearization_KS_OLD} we get for any $(\eta,\xi)\in \sss^1\times \sss^1$
\begin{align}
D_\phi \Big|_{\phi\equiv \lambda} H(\phi)[w] (\eta,\xi) =& D_\phi \Big|_{\phi\equiv \lambda} \left(\frac{\partial u_\phi}{\partial \nu_\phi}\right)[w] (1,\eta,\xi) \\
=& D_\phi \Big|_{\phi\equiv \lambda} \left(\frac{\partial \left(a_\phi + u^\phi \right)}{\partial \nu_\phi}\right)[w] (1,\eta,\xi)\\
=& \frac{1}{\lambda} \frac{\partial}{\partial t}\Big|_{t=1}\left(D_\phi |_{\phi\equiv \lambda} a_\phi [w] \right)(\eta,\xi) +\widetilde{v}''(\lambda) w(\eta, \xi)\\
=& -\frac{\widetilde{v}'(\lambda)}{\lambda} \pder{\varphi^\lambda}{t}(1,\eta,\xi) +\widetilde{v}''(\lambda) w(\eta, \xi)\\
=& \frac{\tan(\lambda)}{2\lambda} \pder{\varphi^\lambda}{t}(1,\eta,\xi)-\frac{1}{2} \frac{1}{\cos^2(\lambda)}w(\eta,\xi).
\end{align}
\end{proof}


\section{Proof of Proposition \ref{Prop_Asymptotics_Sigma_1_Second_Case}}\label{Sec:AppB}

Let $w(\eta,\xi)=w(\eta)$ be one of
\begin{align}
\cos(n\eta) \quad \textnormal{or}\quad \sin(n\eta)
\end{align}
for $n\in \mathbb{N}$. Let $\psi=l(t)w(\eta)$ be the solution to
\begin{align}
\begin{cases}
-\Delta^{g^\lambda} \psi=0 & \inn \Omega\\ \psi=w & \onn \partial \Omega.
\end{cases}
\end{align}
In this case the function $l$ satisfies
\begin{align}
\begin{cases}
\lambda^{-2} l''(t) + \lambda^{-1} \left[\frac{\cos(t\lambda)}{\sin(t\lambda)}-\frac{\sin(t\lambda)}{\cos(t\lambda)} \right] l'(t) - \frac{n^2}{\sin^2(t\lambda)} l(t)=0 & \inn (0,1)\\
l(1)=1.
\end{cases}
\end{align}
As before, let $\theta=t\lambda$ and set $L(\theta):=l\left(\frac{\theta}{\lambda}\right)$. Then
\begin{align}
L'(\theta)=\frac{1}{\lambda} l'\left(\frac{\theta}{\lambda} \right) \quad \andd \quad L''(\theta)=\frac{1}{\lambda^2}l''\left(\frac{\theta}{\lambda} \right)
\end{align}
that provides
\begin{align}\label{Eq:L_ball_2}
\begin{cases}
L''(\theta)+\left[ \frac{\cos(\theta)}{\sin(\theta)}-\frac{\sin(\theta)}{\cos(\theta)}\right] L'(\theta) - \frac{n^2}{\sin^2(\theta)} L(\theta)=0 & \inn (0,\lambda)\\
L(\lambda)=1.
\end{cases}
\end{align}
Fix the solution $L=L^*(\theta)$ regular at $\theta=0$ and normalized so to have $L^*\left(\frac{\pi}{4}\right)=1$ and denote
\begin{align}
B(\theta)=\theta \left[\frac{\cos(\theta)}{\sin(\theta)}-\frac{\sin(\theta)}{\cos(\theta)} \right] \quad \quad \andd \quad \quad C(\theta)=-\theta^2 \frac{n^2}{\sin^2(\theta)}.
\end{align}
Since
\begin{align}
B(0)=1, \quad C(0)=-n^2 \quad \andd \quad L''(\theta)+\frac{B(\theta)}{\theta}L'(\theta) + \frac{C(\theta)}{\theta^2}L(\theta)=0,
\end{align}
it follows that the indicial equation associated to the differential equation of \eqref{Eq:L_ball_2} is
\begin{align}
0=r(r-1)+B(0)r+C(0)=r^2-n^2=(r-n)(r+n)
\end{align}
with roots $n$ and $-n$. Again by Theorem 2 in \cite[Section 7.3]{AD12}, we have two independent Frobenius solutions near $\theta=0$:
\begin{align}
& L_1(\theta)=\theta^n\sum_{j\geq 0} c_j \theta^j\\
&L_2(\theta)=\e L_1(\theta)\ln(\theta)+\theta^{-n}\sum_{j\geq 0} C_j \theta^j \quad \quad (\e \in \{0,1\})
\end{align}
with $c_0\neq 0$ and $C_0\neq 0$. Since $L^*$ is regular at $\theta=0$, it follows that
\begin{align}
L^*(\theta)=C^* L_1(\theta),
\end{align}
where $C^*$ is a constant. Hence $L^*$ is analytic with $L^*(0)=0$ ($n\geq 1$) and, by continuation, $L^*$ can be defined in the whole $\left(0, \frac{\pi}{2} \right)$.


\begin{proposition}[Properties of $L^*$]\label{Prop:Properties_L_ball_2}
Let $n \in \mathbb{N}$. The following properties hold
\begin{enumerate}
\item $L^*>0$ in $\left(0,\frac{\pi}{2} \right)$ and $L^*(0)=0$;
\item $(L^*)'>0$ in $\left( 0,\frac{\pi}{2}\right)$, and hence $L^*\left(\frac{\pi}{2}\right)>0$.
\end{enumerate}
\end{proposition}
\begin{proof}
We already know that $L^*(0)=0$. Since $L^*\left(\frac{\pi}{4}\right)=1$ and since $L^*$ is smooth, the set 
\begin{align}
\mathcal{L}^+:=\left\{\theta \in \left(0,\frac{\pi}{2}\right)\ :\ L^*(\theta)>0\ \ \andd\ \ (L^*)'(\theta)>0\right\}
\end{align}
has nonempty interior. We claim that $\mathcal{L}^+=\left(0,\frac{\pi}{2}\right)$. 

Firstly notice that, by comparison, $0<L^*<1$ in $\left(0,\frac{\pi}{4}\right)$. Since
\begin{align}
\left(\sin(2\theta)(L^*)'\right)'=\frac{n^2 \sin(2\theta)}{\sin^2(\theta)} L^*>0
\end{align}
it follows that $\sin(2\theta)(L^*)'$ is increasing in $\left(0,\frac{\pi}{4}\right)$ and hence $(L^*)'>0$ in $\left(0,\frac{\pi}{4}\right)$, implying that $\left(0,\frac{\pi}{4}\right)\subset \mathcal{L}^+$.

Now let $(0,b)$ be the connected component of $\mathcal{L}^+$ containing $\left(0,\frac{\pi}{4}\right)$. Clearly, $L^*(b)>0$. If $b<\frac{\pi}{2}$, then $(L^*)'(b)=0$ and so $(L^*)''(b)\leq 0$: this contradicts the differential equation of \eqref{Eq:L_ball_2}. It follows that $b=\frac{\pi}{2}$ and hence $\mathcal{L}^+=\left(0,\frac{\pi}{2}\right)$.
\end{proof}


\begin{proposition}[Behaviour of $L^*$ near $\frac{\pi}{2}$]\label{Prop:Behav_L*_ball_2}
Let $n \in \mathbb{N}$. Set $s=\frac{\pi}{2}-\theta>0$ and $\widetilde{L}^*(s)=L^*(\theta)=L^*\left(\frac{\pi}{2}-s\right)$. Then, $\widetilde{L}^*$ satisfies
\begin{align}
\left(\widetilde{L}^*\right)''(s)-\left[\frac{\sin(s)}{\cos(s)}-\frac{\cos(s)}{\sin(s)}\right] \left(\widetilde{L}^*\right)'(s)-\frac{n^2}{\cos^2(s)}\widetilde{L}^*(s)=0
\end{align}
and, as $s\to 0^+$,
\begin{align}
\widetilde{L}^*(s)\sim c_* \ln(s) \quad and \quad \left(\widetilde{L}^*\right)'(s)\sim c_* s^{-1}
\end{align}
for some $c_*<0$.
\end{proposition}
\begin{proof}
Let
\begin{align}
B(s)=-s\left[\frac{\sin(s)}{\cos(s)}-\frac{\cos(s)}{\sin(s)}\right] \quad \quad \andd \quad \quad C(s)=-s^2 \frac{n^2}{\cos^2(s)}.
\end{align}
Then, $B(0)=1$ and $C(0)=0$. The indicial equation is
\begin{align}
0=r(r-1)+B(0)r+C(0)=r^2
\end{align}
whose root is $0$ with multiplicity 2. By Theorem 2 in \cite[Section 7.3]{AD12}, we have two independent Frobenius solutions near $s=0$:
\begin{align}
&L_1(s)=\sum_{j\geq 0} c_j s^j\\
&L_2(s)= L_1(s) \ln(s) + \sum_{j\geq 0} C_j s^j
\end{align}
where $c_0\neq 0$ and $C_0\neq 0$. Since \eqref{Eq:L_ball_2} can be written as
\begin{align}
\left(\sin(2\theta) \left(L^*\right)' \right)'=\frac{n^2 \sin(2\theta)}{\sin^2(\theta)} L^*>0
\end{align}
it follows that $\sin(2\theta) \left(L^*\right)'$ is increasing as $\theta \to \frac{\pi}{2}$. Since $\sin(2\theta)\to 0$ as $\theta \to \frac{\pi}{2}$, then necessarily
\begin{align}\label{Eq:L^*_asympt_2}
\left(L^*\right)'(\theta)\to + \infty \quad \textnormal{as}\ \theta \to \frac{\pi}{2}.
\end{align}
Hence, near $s=0$,
\begin{align}
\widetilde{L}^*(s)=aL_1(s)+bL_2(s)
\end{align}
for some constants $a$ and $b$. In particular, by \eqref{Eq:L^*_asympt_2}, $b\neq 0$. Thus, the claim follows.
\end{proof}

\subsubsection{Behaviour of the eigenvalues}
Fix $w$ as in previous subsection. Then
\begin{align}
D_\phi \Big|_{\phi \equiv \lambda} H(\phi)[w] &=\left(\frac{1}{2\lambda} \tan(\lambda) l'(1) - \frac{1}{2\cos^2(\lambda)} \right)w\\
&=: \sigma_n(\lambda) w.
\end{align}
Set $L(\theta)=\frac{L^*(\theta)}{L^*(\lambda)}$, so to have $L(\lambda)=1$ and $L(\theta)=l\left(\frac{\theta}{\lambda} \right)$. Hence
\begin{align}
L'(\theta)=\frac{1}{\lambda} l'\left(\frac{\theta}{\lambda}\right) \quad \andd \quad L'(\lambda)=\frac{1}{\lambda} l'(1).
\end{align}
Thus
\begin{align}
\sigma_n(\lambda)= \frac{1}{2} \tan(\lambda) \frac{(L^*)'(\lambda)}{L^*(\lambda)}-\frac{1}{2\cos^2(\lambda)}.
\end{align}
By Proposition \ref{Prop:Behav_L*_ball_2}, as $\lambda \to \frac{\pi}{2}$,
\begin{align}
\frac{(L^*)'(\lambda)}{L^*(\lambda)} &=\frac{-(\widetilde{L}^*)'\left(\frac{\pi}{2}-\lambda\right)}{\widetilde{L}^* \left( \frac{\pi}{2}-\lambda\right)}\\
&\sim \frac{- c_*}{c_* \left(\frac{\pi}{2}-\lambda\right) \ln\left(\frac{\pi}{2}-\lambda\right)}\\
&=\frac{-1}{\left(\frac{\pi}{2}-\lambda\right) \ln\left(\frac{\pi}{2}-\lambda\right)}
\end{align}
and hence we conclude that if $n \ge 1$ then
\begin{align}\label{Eq:Asymp_eigenv_infinity_ball_2}
\sigma_n(\lambda)\to -\infty \quad \textnormal{as}\ \lambda\to \frac{\pi}{2}.
\end{align}

On the other hand, consider the function $k(\theta):=\frac{\tan(\theta)}{\tan(\lambda)}$, which satisfies
\begin{align}
\left( \sin(2\theta) k'(\theta)\right)'=\frac{2}{\tan(\lambda)} \frac{1}{\cos^2(\theta)}.
\end{align}
We have that
\begin{align}
\sin(2\theta)\frac{n^2}{\sin^2(\theta)}k(\theta) &=\frac{2}{\tan(\lambda)}n^2 \tan^2(\theta)\\
&= \frac{2n^2}{\tan(\lambda)}
\end{align}
and hence
\begin{align}
\left(\sin(2\theta) k'(\theta) \right)'- \sin(2\theta) \frac{n^2}{\sin^2(\theta)} k(\theta) &=\frac{2}{\tan(\lambda)} \left[\frac{1}{\cos^2(\theta)}- n^2 \right]\\
& \leq \frac{2}{\tan(\lambda)} \left[\frac{1}{\cos^2(\lambda)}-n^2 \right]
\end{align}
in $(0,\lambda)$. Thus, it follows that, if $n\geq 2$ and $0 < \lambda \le \arccos\left(\frac1n\right)$, then
\begin{align}
\left(\sin(2\theta) k'(\theta) \right)'- \sin(2\theta) \frac{n^2}{\sin^2(\theta)} k(\theta) \le 0.
\end{align}
Hence  for such $\lambda$ the function $k$ is a supersolution to the differential equation of \eqref{Eq:L_ball_2} with boundary data $k(0)=0$ and $k(\lambda)=1$. Therefore, by comparison, $L'(\lambda)> k'(\lambda)=\frac{1}{\tan(\lambda)} \frac{1}{\cos^2(\lambda)}$, implying that
\begin{align}
\sigma_n(\lambda) &> \frac{1}{2} \tan(\lambda) \frac{1}{\tan(\lambda) \cos^2(\lambda)}-\frac{1}{2 \cos^2(\lambda)}=0.
\end{align}
Namely, we notice that
\begin{align}\label{a key inequality of the eigenvalue-2}
\mbox{ if } n \ge 2 \mbox{ and } 0 < \lambda \le \arccos\left(\frac1n\right), \mbox{ then } \sigma_n(\lambda) > 0.
\end{align}


\begin{proof}[Proof of Proposition \ref{Prop_Asymptotics_Sigma_1_Second_Case}]\ 
\begin{enumerate}
\item Set $k_n(\lambda):=l'(1)=\lambda \frac{(L^*)'(\lambda)}{L^*(\lambda)}$, which satisfies
\begin{align}
\begin{cases}
k_n'(\lambda)=\left[\tan(\lambda)-\cot(\lambda)+\frac{1}{\lambda}\right]k_n(\lambda) - \frac{1}{\lambda} k_n^2(\lambda) + \frac{n^2\lambda}{\sin^2(\lambda)}\\
k_n(0)=n
\end{cases}
\end{align}
and define $K_n(\lambda):=k_n(\lambda)-\frac{3n}{\cos(\lambda)}$, so that $K_n(0)<0$. In particular, $K_n<0$ in $\left(0,\frac{\pi}{2}\right)$. Indeed, if by contradiction $K_n$ changes sign, then there exists $\lambda^*>0$ minimal so that $K_n(\lambda^*)=0$ and $K_n'(\lambda^*)\geq 0$: hence, using that $k_n(\lambda^*)=\frac{3n}{\cos(\lambda^*)}$,
\begin{align}
0 & \leq K_n'(\lambda^*)\\
& = \left[-\cot(\lambda^*)+\frac{1}{\lambda^*}\right]\frac{3n}{\cos(\lambda^*)} - \frac{1}{\lambda^*} \left(\frac{3n}{\cos(\lambda^*)}\right)^2 + \frac{n^2\lambda^*}{\sin^2(\lambda^*)}\\
& < \left[-\cot(\lambda^*)+\frac{1}{\lambda^*}\right]\frac{9n^2}{\cos(\lambda^*)} - \frac{1}{\lambda^*} \frac{9n^2}{\cos^2(\lambda^*)} + \frac{n^2\lambda^*}{\sin^2(\lambda^*)}\\
&=\frac{\lambda^* \cos^2(\lambda^*) \left[\lambda^* - 9 \sin(\lambda^*)\right] + 9 \sin^2(\lambda^*) \left[\cos(\lambda^*)-1\right]}{\lambda^* \cos^2(\lambda^*)\sin^2(\lambda^*)}n^2\\
&<0
\end{align}
in $\left(0,\frac{\pi}{2}\right)$, obtaining a contradiction. It follows that $k_n(\lambda)<\frac{3n}{\cos(\lambda)}$ which implies that for every $\lambda \in \left(0,\frac{\pi}{2}\right)$
\begin{align}
\frac{\sigma_n(\lambda)}{n}< \frac{3\tan(\lambda)}{2\lambda \cos(\lambda)}-\frac{1}{2n\cos^2(\lambda)}\le \frac{3\tan(\lambda)}{2\lambda\cos(\lambda)}<+\infty.
\end{align}

For the other estimate, consider $g_n(\lambda):=n\lambda \cot(\lambda)$, which satisfies
\begin{align}
\begin{cases}
g_n'(\lambda)< \left[\tan(\lambda)-\cot(\lambda)+\frac 1\lambda\right]g_n(\lambda) - \frac{1}{\lambda} g_n^2 + \frac{n^2 \lambda}{\sin^2(\lambda)}\\
g_n(0)=n
\end{cases}
\end{align}
and define $h_n(\lambda):=k_n(\lambda)-g_n(\lambda)$. One has
\begin{align}
\begin{cases}
h_n'(\lambda)>\left[ \tan(\lambda)-\cot(\lambda)+\frac{1}{\lambda}-\frac{1}{\lambda}(k_n(\lambda)+g_n(\lambda)) \right] h_n(\lambda)\\
h_n(0)=0
\end{cases}
\end{align}
which implies $h_n\geq 0$ in $\left(0,\frac{\pi}{2}\right)$ by the Gronwall's inequality. Whence $k_n\geq g_n$ and so
\begin{align}
\frac{\sigma_n(\lambda)}{n}\geq \frac{1}{2}-\frac{1}{2n\cos^2(\lambda)}\xrightarrow[]{n\to +\infty} \frac{1}{2}.
\end{align}


\item For $i <j$ and define $h(\lambda):=k_j(\lambda)-k_i(\lambda)$:
\begin{align}
\begin{cases}
h'(\lambda) > \left[\tan(\lambda)-\cot(\lambda)+\frac{1}{\lambda}-\frac{1}{\lambda}(k_i(\lambda)+k_j(\lambda))\right] h(\lambda)\\
h(0)=j-i>0
\end{cases}
\end{align}
which implies, by Gronwall's inequality, $h>0$ in $\left(0,\frac{\pi}{2}\right)$ and hence $k_j>k_i$. So $\sigma_j>\sigma_i$ in $\left(0,\frac{\pi}{2}\right)$.

\item Both (a) and (b)  follow from what seen in this section.

\item The existence of $\lambda_n$ follows by point (3). For the uniqueness, if $\lambda_n\in \left(0,\frac{\pi}{2}\right)$ is such that $\sigma_n(\lambda_n)=0$, then
\begin{align}
L'(\lambda_n)=\frac{1}{\sin(\lambda_n)\cos(\lambda_n)}.
\end{align}
Using that $L'(\lambda)=\frac{(L^*)'(\lambda)}{L^*(\lambda)}$ and the equation satisfied by $L^*$, one gets
\begin{align}\label{sigma'_lambda_n_2}
\sigma'_n(\lambda_n)&=\frac{n^2}{2\sin(\lambda_n)\cos(\lambda_n)}-\frac{1}{2\sin(\lambda_n)\cos^3(\lambda_n)}\\
&=\frac{1}{2 \sin(\lambda_n)\cos(\lambda_n)}\left(n^2-\frac{1}{\cos^2(\lambda_n)}\right).
\end{align}
Then, since $\sigma_n(\lambda_n)=0$, the fact that  $\sigma_n'(\lambda_n)<0$ follows from \eqref{a key inequality of the eigenvalue-2}.
This proves both (a) and the uniqueness of $\lambda_n$. Moreover, \eqref{a key inequality of the eigenvalue-2} yields that $\lambda_n \in \left(\arccos\left(\frac 1n\right),\frac \pi 2\right)$, which proves (c).
\end{enumerate}
\end{proof}

\bibliography{references}
\bibliographystyle{abbrv}
\end{document}